\documentclass[12pt]{article}
\usepackage{amsmath,amsfonts,amsthm,amssymb}
\usepackage{ifthen}
\usepackage{fullpage}
\usepackage{enumitem}

\usepackage{epic,eepic}
\usepackage{graphicx}
\usepackage{amsbsy}

\newcommand\comment[1]{}

\newtheorem{theorem}{Theorem}[section]
\newtheorem{lemma}[theorem]{Lemma}

\newcommand\F{{\mathcal F}}
\newcommand\fw{\mathop{\mathrm{fw}}}
\newcommand\nsfw{\mathop{\mathrm{nsfw}}}
\newcommand\leqm{\leq_{\mathrm{m}}}     % minor relation
\newcommand\Int{\mathop{\mathrm{Int}}}
\newcommand\bnd{\partial}     % boundary

\begin{document}

\title{Homological face-width condition forcing $K_6$-minors in graphs on surfaces}

%\author{Roi Krakovski \and Bojan Mohar}

\author{
   Roi Krakovski\thanks{Supported in part by postdoctoral support at the Simon Fraser University.}\\[1mm]
   {Department of Mathematics}\\{Simon Fraser University}\\{Burnaby, B.C. V5A 1S6}
\and
   Bojan Mohar\thanks{Supported in part by an NSERC Discovery Grant (Canada), by the Canada Research Chair program, and by the Research Grant P1--0297 of ARRS (Slovenia).}~~\thanks{On leave from: IMFM \& FMF, Department of Mathematics, University of Ljubljana, Ljubljana, Slovenia.}\\[1mm]
   %~~\thanks{Email: {\tt mohar@sfu.ca}}
   {Department of Mathematics}\\{Simon Fraser University}\\{Burnaby, B.C. V5A 1S6}
   }

\maketitle

\begin{abstract}
It is proved that every graph embedded on a (non-spherical) surface with non-separating face-width at least $7$ contains a minor isomorphic to $K_6$.
It is also shown that face-width four yields the same conclusion for graphs on the projective plane.
\end{abstract}

\section{Introduction}

A \emph{surface} is a connected compact 2-manifold. Unless explicitly stated otherwise, surfaces will be assumed to be non-simply connected and have no boundary.
If there is a nonempty boundary, then we speak of a \emph{bordered surface} and every component of the boundary is called a \emph{cuff}. A simple closed curve $\gamma$ on a surface $\Sigma$ is said to be \emph{surface separating} or \emph{zero-homologous} if cutting $\Sigma$ along $\gamma$ results in a disconnected (bordered) surface. Two disjoint simple closed curves are said to be \emph{homologous} if they are either both zero-homologous, or none of them is zero-homologous, but cutting the surface along both of these curves disconnects the surface.

Let $G$ be a graph embedded on a surface $\Sigma$. We regard $G$ as a subset of $\Sigma$ (that is, we identify $G$ with its embedding on $\Sigma$). The \emph{face-width} of $G$, denoted by
$\fw(G)$, is the maximum number $k$ so that every non-contractible simple closed curve in
$\Sigma$ intersects $G$ in at least $k$ points.
The homology version, the \emph{non-separating face-width} of $G$, denoted by $\nsfw(G)$, is the maximum number $k$ so that every surface non-separating
simple closed curve in $\Sigma$ intersects $G$ in at least $k$ points.
We refer to \cite{mohar} for additional terminology involving graphs embedded in surfaces.

A graph $H$ is a \emph{minor} of a graph $G$, in symbols $H \leqm G$, if $H$ can be obtained from a subgraph of $G$ by a series of contractions of edges.

The theory of graph minors (Robertson and Seymour \cite{RS13}) shows that for every surface $\Sigma$ there exists a constant $c_{\Sigma}$ (depending only on $\Sigma$) such that if $G$ embeds in $\Sigma$ with face-width at least $c_\Sigma$, then $G$ contains $K_6$ as a minor.
We are interested in finding the best possible value for $c_{\Sigma}$. If $G$ is an apex graph, then $G$ does not contain $K_6$ as a minor. It is known that there are apex graphs that can be embedded on non-spherical surfaces with face-width at least three, see \cite{Mo97}. Hence, there are surfaces $\Sigma$ with $c_\Sigma \geq 4$. In fact, there are examples showing that $c_\Sigma\ge4$ for every surface $\Sigma$. We first show that $c_\Sigma=4$ in the special case when $\Sigma$ is the projective plane.

\begin{theorem}
\label{thm:pp}
Let $G$ be a graph embedded on the projective plane. If\/ $\fw(G) \geq 4$, then $K_6 \leqm G$.
\end{theorem}

We suspect that $c_\Sigma = 4$ for every $\Sigma$; however a proof (or disproof) seems to be out of reach. Our main result given below extends Theorem \ref{thm:pp} to arbitrary surfaces and strengthens the afore-mentioned result of Robertson and Seymour from \cite{RS13} in two ways. First, we obtain an upper bound on $c_{\Sigma}$ that is independent of the surface. In addition to this, we are able to loosen the face-width requirement by involving a condition on the non-separating face-width instead. Note that for graphs on the projective plane, we have $\nsfw(G)=\fw(G)$ and that $\nsfw(G)\ge\fw(G)$ holds in general.

\begin{theorem}
\label{thm:main}
Every graph $G$ embedded on a non-spherical surface with
$\nsfw(G) \geq 7$ contains the complete graph $K_6$ as a minor.
\end{theorem}

There is a continuing interest in the structure of graphs that do not contain $K_6$ as a minor.
An outstanding open problem in this area is a conjecture of J{\o}rgensen \cite{Jorg} that every 6-connected graph has no $K_6$-minor if and only if it can be made planar by removing one vertex. An asymptotic version of J{\o}rgensen's Conjecture has been recently proved by Kawarabayashi et al.\ \cite{KNTW}. The known structure of such graphs is used in \cite{KKMR} in the design of an efficient algorithm for constructing linkless embeddings of graphs in 3-space. As for graphs embedded in surfaces, several papers \cite{KMN,MN,NOO,MNOS} concern $K_6$-minors in triangulations of surfaces of small genus, while \cite{FM} obtained a general result about $K_6$-minors in graphs in the projective plane.

All graphs in this paper are finite and simple. Paths and cycles have no ``repeated vertices''. A path $P=x_0x_1\dots x_n$ is given by the sequence of its consecutive vertices $x_0,x_1,\dots,x_n$, but it is considered as a subgraph. If a path $P$ has endvertices $u$ and $v$, then $P$ is called a \emph{$(u,v)$-path} (also \emph{$(v,u)$-path}). The \emph{order} of a path $P$, denoted as $|P|$, is its number of vertices.
For vertices $a$ and $b$ on a path $P$, $P[a,b]$ denotes the $(a,b)$-path contained in $P$, and $P[a,b)=P[a,b] - b$ denotes the path from $a$ to the predecessor of $b$. The paths $P(a,b]$ and $P(a,b)$ are defined analogously.
The same notation is used for cycles with given clockwise orientation, where $C[a,b]$ denotes the path from $a$ to $b$ in the clockwise direction.

For $A_i \subseteq V(G)$ or $A_i \subseteq G$ ($i=1,2$), an \emph{$(A_1,A_2)$-path\/} is an $(a_1,a_2)$-path $P$ with $V(P)\cap V(A_i)=\{a_i\}$ for $i=1,2$, an \emph{$(A_1)$-path\/} is an $(a_1,a_2)$-path with $V(P) \cap V(A_1) =\{a_1,a_2\}$, where $a_1\ne a_2$ and $P$ contains an edge that is not in $A_1$.

\section{Face-chains}
\label{sec:FaceChains}

Let $G$ be a graph embedded in a surface $\Sigma$. We denote by $F(G)$ the set of all \emph{facial walks} of $G$. Each facial walk is also considered as being a subgraph of $G$ consisting of all vertices and edges on the boundary of a face of the embedding. The open face corresponding to the facial walk $F$ will be denoted by $F^\circ$.

Let $n\geq 0$ be an integer. A \emph{face-chain} $\Lambda$ of \emph{length} $n$ is an alternating sequence $x_0, F_0, x_1,\dots, \allowbreak x_{n-1},F_{n-1},x_n$ such that, for $i=0,\dots,n-1$, $F_i \in F(G)$ and $x_i,x_{i+1} \in V(F_i)\cup E(F_i)$. Note that $x_{i+1}$ is either a vertex or an edge in $F_i\cap F_{i+1}$. We also write $|\Lambda|=n$ to denote the length of $\Lambda$. If $x_0=x_n$, then the face-chain is said to be \emph{closed}.  We define $X(\Lambda) = \{x_0,\dots,x_n\}$ and $G(\Lambda) = \bigcup_{i=0}^{n-1} F_i \subseteq G$.

Let $\Lambda=x_0,F_0,x_1,\dots,x_{n-1},F_{n-1},x_0$ be a closed face-chain.
We define a closed curve $\Gamma(\Lambda) \subseteq \Sigma$ by taking the composition of simple arcs in each $F_i$ joining $x_i$ and $x_{i+1}$. (Note that the choice of a simple arc in $F_i$ is determined up to homotopy, if we assume that every $F_i$ is homeomorphic to an open disk and that each of $x_i$ and $x_{i+1}$ appears in the facial walk $F_i$ just once; these assumptions will always be satisfied.)

%Conversely, a simple closed curve $\Gamma$ on $\Sigma$ having $n$ points in common with $G$ determines an alternating sequence $x_0,F_0, x_1,\dots,x_{n-1},F_{n-1},x_0$ of elements of $G$ (where $x_i \in E(G)\cup V(G)$ and $F_i \in F(G)$), such that $\Gamma \cap G=\{x_0,\dots,x_{n-1}\}$ and $\Gamma \subseteq \left( \bigcup_{i=0}^{n-1} F_i^\circ \right) \cup \{x_0,\dots,x_{n-1}\}$.

We say that a face-chain $\Lambda$ is \emph{nice}, if for all $0\leq i<j \leq n-1$, we have $F_i \neq F_j$, $x_i \neq x_j$, and $x_k\ne x_n$ for $1\le k < n$. Note that if $\Lambda$ is nice then $\Gamma(\Lambda)$ is simple. A nice face-chain $\Lambda$ is \emph{clean} if for $i=0,\dots,n-1$, $F_i \cap F_{i+1}=\{x_{i+1}\}$ (where $F_n=F_0$) and for all $0 \leq i<j \leq n-1$, with $1 \neq j-i \neq n-1$, we have $F_i\cap F_j =\emptyset$.

To avoid repetition let us state the following assumption together with its notation, since it will be common to several statements that follow.

\bigskip

\noindent
\textbf{(H1)} For a graph $G$ embedded in a surface $\Sigma$, let $\Lambda=x_0,F_0,x_1,\dots,F_{n-1},x_{0}$ be a closed face-chain of length $n$ such that:
\begin{enumerate}
	\item[(i)] $\Gamma(\Lambda)$ is surface non-separating;
	\item[(ii)] subject to (i), $|\Lambda|$ is minimum.
\end{enumerate}

\medskip

We shall abuse terminology and call a face-chain $\Lambda$ surface separating or contractible when $\Gamma(\Lambda)$ has that property.

The following result is well-known (cf.~\cite{mohar}) and is referred to as the \emph{$3$-path condition}.

\begin{theorem}
\label{3-path-new}
Let $G$ be a graph embedded on $\Sigma$, and let $x,y\in V(G)$. Suppose $G$ contains three $(x,y)$-paths, $P_1,P_2,P_3$, pairwise disjoint except for their ends. Let $C_{ij}$ $(1 \leq i<j \leq 3)$ be the cycle $P_i \cup P_j$. Then the following holds:
\begin{enumerate}
\item[\rm (a)] If two of the three cycles $C_{ij}$ are contractible, then  so is the third.
\item[\rm (b)] If two of the three cycles $C_{ij}$ are surface separating, then  so is the third.
\end{enumerate}
\end{theorem}

Let $\Lambda$ be a closed face-chain. Let $\Lambda'=w_0,F_1',w_1,\dots,w_{k-1},F_{k}',w_{k}$ be a face-chain (not closed) of length $k$ such that $w_0$ is incident with a face $F_i$ and $w_{k}$ is incident with a face $F_j$ ($0 \leq i < j \leq n-1$) in the face-chain $\Lambda$. There are two face-chains in $\Lambda$ whose first and last faces are $F_i$ and $F_j$. We can combine each of these with $\Lambda'$ to get a closed face-chain containing $\Lambda'$.
By using the 3-path property, we deduce the following.

\begin{theorem}
\label{maintool}
Let $G$ be a graph embedded in $\Sigma$, and let $\Lambda$ be as in {\rm (H1)}. Let $\Lambda'=w_0,F_1',w_1,\dots,w_{k-1},F_{k}',w_{k}$ be a face chain of length $k \geq 0$, where $w_0$ and $w_k$ vertices or edges that are incident with faces $F_i$ and $F_j$ in $\Lambda$ $(0\leq i<j \leq n-1)$. Then the closed face-chain formed by $\Lambda'$ and the shorter one of the two face-chains from $F_i$ to $F_j$ in $\Lambda$ is a face-chain of length $\leq 2k+2$.
\end{theorem}

\begin{proof}
Let $\Lambda_1=F_i,x_{i+1},F_{i+1},\dots,x_j,F_j$ and $\Lambda_2=F_j,x_{j+1},F_{j+1},\dots,\allowbreak x_0,F_0,\allowbreak \dots,x_i,F_i$ be the two face-subchains from $F_i$ to $F_j$ contained in $\Lambda$. They together use $n+2$ faces. Let us now consider the two closed face-chains $\Lambda_1 \cup \Lambda'$ and $\Lambda_2 \cup \Lambda'$. Clearly,
$$
   |\Lambda_1 \cup \Lambda'| + |\Lambda_2 \cup \Lambda'| = n + 2 + 2k.
$$
By the 3-path condition, at least one of them is surface non-separating, thus it is of length at least $n$ by (H1)(ii). So, it follows that the length of the other one is at most $2k+2$.
\end{proof}

The following theorem is well-known for the face-width (cf.~\cite{mohar}); the proof for non-separating face-width is essentially the same.

\begin{theorem}
\label{fw-lemma}
Let $G$ be  a 3-connected graph embedded on a surface $\Sigma$ with $\nsfw(G) \geq 3$. Then all facial walks of $G$ are cycles, and any two of them are either disjoint or intersect in a
single vertex or a single edge.
\end{theorem}

The following is an easy corollary of Theorem \ref{fw-lemma}, Theorem \ref{maintool} (with $k=0$) and the 3-path condition.

\begin{theorem}
\label{clean}
Let $G$ be  a 3-connected graph embedded on a surface $\Sigma$ with $\nsfw(G) \geq 3$, and let $\Lambda$ be as in {\rm (H1)}. Then $\Lambda$ is clean.
\end{theorem}

The following result is an easy corollary of the 3-path condition. Its proof for the edge-width can be found in \cite{Thomassen93}; for the proof of the face-width version, see \cite{mohar}; the proof for the non-separating face-width is essentially the same as in \cite{mohar}.

\begin{theorem}
\label{newfacewidth}
Let $G$ be embedded in a surface $\Sigma$, and let $\Lambda$ be as in {\rm (H1)}. Let $G'$ be obtained from $G$ by cutting $\Sigma$ along $\Gamma(\Lambda)$ and capping off the resulting cuffs. Then $\fw(G') \geq  \left\lceil  \frac{1}{2} \fw(G) \right\rceil$ and\/ $\nsfw(G') \geq  \left\lceil  \frac{1}{2} \nsfw(G) \right\rceil$.
\end{theorem}

Let $G$ be a graph embedded on $\Sigma$ and let $p\in\Sigma\setminus G$ be a preselected point on the surface. If $C$ is a surface-separating cycle of $G$, we denote by $\Int(C)$ the subgraph of $G$ contained in the part of the surface separated by $C$ that contains $p$; in particular, $C\subseteq \Int(C)$.
Let $f\in F(G)$ be a face of $G$. We define subgraphs $B_0(f), B_1(f),B_2(f),\dots$ of $G$
recursively as follows: $B_0(f)=f$, and for $k \geq 1$, $B_k(f)$ is the union of $B_{k-1}(f)$ and
all facial walks that have a vertex in $B_{k-1}(f)$. Let $\partial B_k(f)$ be the set of
edges of $B_k(f)$ (together with their ends) that are not incident with a vertex of $B_{k-1}(f)$.
With this notation we have the following result (see \cite{Mo92}).

\begin{theorem}
\label{fw-lemma-1}
Let $G$ be a graph embedded on $\Sigma$ with $\nsfw(G) \geq 2$.
Let $f \in F(G)$ and let $k=\left\lfloor \tfrac{1}{2}\nsfw(G)\right\rfloor-1$. Then there exist pairwise
disjoint surface-separating cycles $C_0(f),\dots,C_k(f)$ such that for $i=0,\dots,k$, $C_i(f) \subseteq \partial B_i(f)$ and $B_i(f) \subseteq \Int(C_i(f))$ (where $\Int$ is defined with respect to a point $p$ in the face $f$). Moreover, if $l=\left\lfloor \tfrac{1}{2}\fw(G)\right\rfloor-1$, then the cycles $C_0,C_1,\dots,C_l$ are contractible in $\Sigma$.
\end{theorem}

In addition to having large $\nsfw(G)$, we will also need $\fw(G)$ to be large. This will be made possible by the following theorem.

\begin{theorem}
\label{nsfw and large fw}
Let\/ $G$ be a graph embedded in a surface $\Sigma$ with $k = \nsfw(G)\geq 6$.
Then $G$ contains a minor $G'$ such that $G'$ is 3-connected and has an embedding in a surface $\Sigma'$ with\/ $\nsfw(G') = k$ and\/ $\fw(G')\ge 6$.
\end{theorem}

\begin{proof}
Let $G'$ be a minor of $G$ with the minimum number of vertices and edges such that $G'$ has an embedding in a surface $\Sigma'$ with\/ $\nsfw(G') = k$. Clearly, $G'$ exists. We claim that $\fw(G')\ge 6$. If not, let $1\le l\le 5$ be the smallest integer such that there exists a closed face-chain $\Lambda = x_0,F_0,x_1, \dots, F_{l-1},x_0$ with $\Gamma(\Lambda)$ non-contractible. Since $k\ge 6$, $\Gamma(\Lambda)$ is surface-separating. Let $\Sigma_1'$ and $\Sigma_2'$ be the surfaces obtained by cutting $\Sigma'$ along $\Gamma(\Lambda)$ and capping off the resulting cuff. For $i=1,2$, let $G_i'$ be the subgraph of $G'$ in $\Sigma_i'$, and let $G_i''$ be obtained from $G_i'$ by adding a vertex of degree $l$ and joining it to all vertices in $X=\{x_0,\dots,x_{l-1}\}$ (and embedding the vertex and these edges into the capped disk). Each face $F_j$ ($0\le j<l$) determines a face $F^i_j$ in $\Sigma_i'$. This correspondence makes it possible to convert every face-chain in $G_1''$ to a face-chain in $G'$. Note that $G_1''$ is a proper minor of $G'$ (since $l$ is smallest, $G_2'\setminus X$ contains a connected component adjacent to all vertices in $X$ and can thus be contracted into the added vertex of $G_1''$). By the minimality of $G'$, we conclude that the embedding of $G_1''$ in $\Sigma_1'$ has $\nsfw(G_1'')<k$. Let $\Lambda'$ be a non-separating closed face-chain of length $k' < k$ confirming this fact. It is easy to see that $\Lambda'$ determines a non-separating face-chain in $G'$ of the same length (since $l\le 5$). This contradiction proves that $\fw(G')\ge 6$.

Finally, since $\fw(G')\ge 3$, $G'$ contains a 3-connected minor whose face-width and non-separating face-width are the same (\cite{mohar}). By the minimality of $G'$, this minor is equal to $G'$. This completes the proof.
\end{proof}

\section{Disjoint paths on a surface}
\label{sec:DisjointPathsOnASurface}

Let $G$ be a graph embedded on $\Sigma$. Let $C_1,C_2 \subseteq G$ be disjoint, homologous, surface non-separating cycles in $G$. Note that $C_1$ and $C_2$ are 2-sided since pairs of 1-sided homologous cycles always intersect each other.
Let $\Sigma_0$ and $\Sigma_1$ be bordered surfaces, whose cuffs coincide with $C_1$ and $C_2$, where $\Sigma_0 \cup \Sigma_1 = \Sigma$, $\Sigma_0 \cap \Sigma_1 = C_1 \cup C_2$.
Similarly, we can write $G=G_0\cup G_1$, where $G_i$ is the subgraph of $G$ embedded in $\Sigma_i$ for $i=0,1$, and thus
$G_0 \cap G_1 = C_1 \cup C_2$.
For $i=0,1$, we denote by $\Sigma_i'$ the closed surface obtained from the bordered surface $\Sigma_i$ by capping off the two cuffs of $\Sigma_i$.
With this notation we have the following:

\begin{lemma}
\label{lem:disjoint-paths}
Each of\/ $G_0$ and\/ $G_1$ contains $\nsfw(G)$ pairwise disjoint $(C_1,C_2)$-paths.
\end{lemma}

For the proof of Lemma \ref{lem:disjoint-paths}, we need the following result whose weaker form for contractible curves has appeared in \cite{BMT}.
%Although the proof is essentially the same as that given in \cite{BMT}, we include it here for the sake of completeness.

\begin{lemma}
\label{bohmemoharcarsten}
Let $G$ be a  graph embedded on a surface $\Sigma$, and let $A$ (possibly $A=\emptyset$) be a set of vertices such that $G'=G-A$ is disconnected. Let $\hat{C}_1$ and $\hat{C}_2$ be distinct connected components of\/ $G'$. Then $\Sigma$ contains a simple closed curve $\Gamma$ such that\/ $\Gamma \cap G \subseteq A$ and if\/ $\Gamma$ is surface separating, then $\hat{C}_1$ and $\hat{C}_2$ are contained in different connected components of\/ $\Sigma \setminus \Gamma$.
\end{lemma}

\begin{proof}
Consider the disconnected graph $G'$ with its induced embedding on $\Sigma$. We claim that
\begin{enumerate}
\item [\rm(1)] $\Sigma \setminus G'$ contains a 2-sided simple closed curve $\Gamma$ that intersects $G$ only in edges
 joining $\hat{C}_1$ with $A$, and $\Gamma$ is either surface non-separating in $\Sigma$, or separates $\Sigma$ into two components, one containing $\hat{C}_1$ and the other one containing $\hat{C}_2$.
\end{enumerate}
To see this, let us first delete all components of $G'$ distinct from $\hat{C}_1$ and $\hat{C}_2$. Next, let us add an edge $e$ joining a vertex in $\hat{C}_1$ with a vertex in $\hat{C}_2$ so that the resulting graph $G'' = \hat{C}_1\cup \hat{C}_2 +e$ is embedded in $\Sigma$.
Since $e$ is a cut-edge of $G''$, the unique facial walk $F$ containing $e$ in the induced embedding of $G''$ contains $e$ twice and $e$ is traversed in opposite directions. Following the part of this facial walk in $\hat{C}_1$, we see that $\Sigma$ contains a simple closed curve $\Gamma$ that follows the boundary of $F$ close to $\hat{C}_1$ so that $\Gamma$ crosses $e$ exactly once, and $\Gamma$ intersects only $e$ and the edges of $G$ joining $\hat{C}_1$ with $A$. In particular, $\Gamma$ does not intersect any of the removed components of $G'$. If $\Gamma$ separates $\Sigma$, then each component of $\Sigma\setminus \Gamma$ contains exactly one of the components $\hat{C}_1$ or $\hat{C}_2$ since the edge $e$ crosses $\Gamma$. This proves~(1).

Let us consider all simple closed curves satisfying the conclusion of (1), except that we allow them to intersect $G$ not only at interior points of the edges joining $\hat{C}_1$ with $A$, but also allow that $\Gamma$ passes through vertices in $A$. Among all such curves,
choose $\Gamma \subseteq \Sigma$ having minimum number of crossings with interior points on the edges joining $\hat{C}_1$ with $A$. Note that $\Gamma$ intersects $G$ only in $A$ or in edges joining  $A$ to vertices in $\hat{C}_1$. By possibly altering $\Gamma$, we may assume that each intersection of $\Gamma$ with an edge of $G$ is a crossing.

If $\Gamma \cap E(G) = \emptyset$, then $\Gamma$ is of the desired form and the claim follows. Hence $\Gamma$ intersects an edge $a=uv \in E(G)$, where $u\in V(\hat{C}_1)$ and $v \in A$. Replace a short segment of $\Gamma$ around this intersection with a simple curve which follows $a$ to its endvertex $v$ in $A$, crosses through $v$ and returns back on the other side of $a$ (if $\Gamma$ intersects $a$ in more than one point, choose the intersection point which is closest to $v$). The resulting curve $\Gamma'$ is homotopic to $\Gamma$. By the minimality property of $\Gamma$, $\Gamma'$ is not simple, and is hence composed of two simple closed curves $\Gamma_1$ and $\Gamma_2$ that intersect at $v$.

We may assume that both $\Gamma_1$ and $\Gamma_2$ separate $\Sigma$, for if  $\Gamma_i$ ($1 \leq i \leq 2$)  does not separate $\Sigma$, then $\Gamma_i$ can be chosen instead of $\Gamma$, contradicting the minimality of $\Gamma$.
By our choice of $\Gamma$ and $a$, it is easy to see that there must exist $i\in \{1,2\}$ such that cutting $\Sigma$ along $\Gamma_i$ disconnects $\Sigma$ into two components each containing exactly one component $\hat{C}_1$ or $\hat{C}_2$. But then  $\Gamma_i$ can be chosen instead of $\Gamma$, contradicting the minimality of $\Gamma$. This completes the proof.
\end{proof}

\begin{proof}[Proof of Lemma \ref{lem:disjoint-paths}]
By symmetry, it suffices to prove the lemma for $G_1$.
Let $r$ be the maximum number of disjoint $(C_1,C_2)$-paths contained in $G_1$. To prove the claim, we have to show that $r\ge\nsfw(G)$. By Menger's theorem there exists $A \subseteq V(G_1)$ with $|A|=r$ that separates $C_1$ and $C_2$.
Let $G_2\supseteq G_1$ be the graph embedded in $\Sigma_1'$ that is obtained from $G_1$ by adding two vertices $v_1,v_2$, where $v_i$ is adjacent to all vertices in $C_i$ ($i=1,2$). For $i=1,2$, let $\hat{C}_i$ be the connected component of $G_2-A$ containing $v_i$. Then $A$ satisfies assumptions of Lemma \ref{bohmemoharcarsten}.
Let $\Gamma$ be a simple closed curve on $\Sigma_1'$ as promised to exist by Lemma \ref{bohmemoharcarsten}.
If $\Gamma$ is surface-separating in $\Sigma_1'$, then it separates $v_1$ from $v_2$.
Moreover, $\Gamma \cap G_2 \subseteq A$ and we may assume that $\Gamma$ is disjoint from the interior of $\Sigma_0$.
However, in the surface $\Sigma$, $\Gamma$ is surface non-separating since the two components of $\Sigma_1'\setminus \Gamma$ are connected together in $\Sigma\setminus\Gamma$ through $\Sigma_0$. Thus, we conclude that $\Gamma$ is always surface non-separating. Therefore, $r\ge |\Gamma\cap G| = |\Gamma\cap G_2| \ge \nsfw(G)$.
\end{proof}

For a path $P$, we denote by $int(P)$ the path obtained from $P$ by removing its end-vertices (and incident edges).

\begin{theorem}
\label{thm:disjoint-paths-2}
Let $G,C_1,C_2$ and $G_1,\Sigma_0,\Sigma_1$ be as introduced at the beginning of the section.  Let $\mathcal{P}$ be a set of pairwise disjoint $(C_1,C_2)$-paths in $G_1$ of maximum cardinality. Let $i \in \{1,2\}$, and let $w,w' \in V(C_i)$ be two vertices of\/ $C_i$.
Let $X, \overline{X} \subseteq C_i$ be the two $(w,w')$-paths on $C_i$, i.e.,\ $C_i=X\cup \overline{X}$ and $X\cap \overline{X} = \{w,w'\}$. Then one of the following holds:
\begin{enumerate}
\item[\rm (a)] There exists an $(int(\overline{X}),\mathcal{P})$-path in $G_1$ disjoint from $X$.
(Here we consider $\mathcal{P}$ as a subgraph of $G$.)
\item[\rm (b)] There exist $v,u \in V(X)$ such that $v$ and $u$ are incident with a common face
$f \in F(G_1)$ in $\Sigma_1$, and the closed curve in $\Sigma$ formed by a simple arc in $f^\circ$ from $u$ to $v$ together with the segment $X[v,u]$ on $X$ is %non-contractible
surface non-separating in $\Sigma$.
\end{enumerate}
\end{theorem}

\begin{proof}
By symmetry we may assume that $i=1$. We may also assume that $int(\overline{X})\ne \emptyset$ since otherwise (b) holds with $\{v,u\} = \{w,w'\}$. Moreover, no path in $\mathcal P$ has an end in $int(\overline{X})$ since otherwise (a) holds.
We will assume that (a) fails, and show that (b) holds. In particular, we will show that there exists a simple arc $\gamma$ in $\Sigma_1$, so that $\gamma \cap G_1 = \{v,u\}$, where $v,u \in V(X)$, and $\gamma \cup X[v,u]$ is a surface non-separating closed curve. This will imply (b).

So, suppose (a) does not hold.
The maximality of $|\mathcal{P}|$ and the assumption that (a) does not hold, imply that in $G_1-V(X)$, there is no $(int(\overline{X}), \mathcal{P}\cup C_2)$-path. Hence, $G_1-V(X)$ is disconnected, with $C_2$ and $int(\overline{X})$ belonging to distinct connected components.

Let $G_2$ (embedded on $\Sigma_1'$) be obtained from $G_1-V(X)$ by adding an edge $e$ (embedded along the deleted path $X$, but drawn inside the capped face in $\Sigma_1'$ so that it does not intersect $G_1$) connecting the two end vertices of $int(\overline{X})$. Note that $G_2$ has the same connected components as $G_1-V(X)$, since $e$ connects two vertices of $int(\overline{X})$ that are in the same component of $G_1-V(X)$.
Let $F_1$ be the face of $G_2$ in $\Sigma_1'$ bounded by the cycle $C_1'=int(\overline X) + e$.

Clearly, $C_1'$ is a cycle in $G_2$ which is homotopic to $C_1$ in $\Sigma$. Let $\hat{C}_1'$ and  $\hat{C}_2$ be the connected components of $G_2$ containing $C_1'$ and $C_2$, respectively. Let $\Gamma \subseteq \Sigma_1'$ be the closed curve obtained by applying Lemma \ref{bohmemoharcarsten} to the embedded graph $G_1+e\subseteq \Sigma_1'$, the separating vertex set $V(X)$ playing the role of $A$,
and considering the connected components $\hat{C}_1'$ and $\hat{C}_2$ of $G_2 = (G_1+e)-V(X)$. Then $\Gamma \cap int(F_1) =\emptyset$ since $\Gamma\cap \hat{C}_1' = \emptyset$. In particular, $\Gamma \subseteq \Sigma_1$ and $\Gamma\cap G_1\subseteq V(X)$.
We claim that

\begin{itemize}
	\item [(1)] $\Gamma$ is surface non-separating in $\Sigma$.
\end{itemize}
This is clear if $\Gamma$ is surface non-separating in $\Sigma_1'$. Otherwise, $\Gamma$ is surface separating in $\Sigma_1'$. As guaranteed by the use of Lemma \ref{bohmemoharcarsten}, $\Gamma$ separates $\hat{C}_1'$ from $\hat{C}_2$ in $\Sigma_1'$. However, in $\Sigma$, these two parts are connected together via the surface part $\Sigma_0$, so $\Gamma$ is not surface separating in $\Sigma$. This proves (1).

In the sequel we will consider curves in $\Sigma_1'\setminus int(F_1)$. We can view $\Sigma_1 \subset \Sigma_1'\setminus int(F_1) \subset \Sigma$ and therefore talk about homology properties of such curves in $\Sigma$.

Let $\Gamma_1$ be a closed curve in $\Sigma_1' \setminus int(F_1)$ so that the following conditions hold:
\begin{itemize}
	\item [(i)] $\Gamma_1 \cap (G_1+e) \subseteq V(X)$, $\Gamma_1$ is surface non-separating in $\Sigma$, and every arc $\gamma \subseteq \Gamma$ with ends $x,y \in V(X)$ and which is otherwise disjoint from $\Sigma_1$ is homotopic to $X[x,y]$;
	\item [(ii)] subject to (i), the number of connected components of $\Gamma_1\cap \Sigma_1$ is minimum.
\end{itemize}

Note that such a choice of $\Gamma_1$ is possible since $\Gamma$ satisfies (i).

The curve, $\Gamma_1$ is surface non-separating in $\Sigma$. By (i) we deduce that there exists an arc $\gamma \subseteq \Gamma_1$ contained in $\Sigma_1$ such that $\gamma\cap G_1 =\{x,y\}$ where $x,y \in V(X)$. Let $\gamma'$ be a curve in $\Sigma_1'\setminus \Sigma_1$ along $X$ with ends $x$ and $y$ that is homotopic to $X[x,y]$.

If $\gamma \cup \gamma'$ is surface non-separating in $\Sigma$, then (b) holds. Otherwise, by replacing $\gamma$ in $\Gamma_1$ with $\gamma'$, we obtain a new curve $\Gamma_2$ that satisfies (i), by the 3-path condition. But then the existence of $\Gamma_2$ contradicts (ii) in the choice of $\Gamma_1$. This contradiction concludes the proof.
\end{proof}

The following result is a well-known corollary of Menger's theorem.

\begin{lemma}
\label{lem:cylinder}
Let\/ $\Sigma$ be a cylinder and let $F_1$ and $F_2$ be the two cuffs. Let $G$ be a graph embedded on $\Sigma$ and suppose that for $i=1,2$, $S_i:=F_i \cap G \subseteq V(G)$. Let $r \geq 0$ be an integer. Suppose that every simple closed curve $\Gamma$ with $\Gamma \cap G \subseteq V(G)$ and $|\Gamma \cap V(G)|<r$ is contractible in $\Sigma$. Then there are $r$ pairwise disjoint $(S_1,S_2)$-paths in $G$.
\end{lemma}

%%%%%%%%%%%%%%%%%%%%%%%%%%%%%%%%%%%%%%%%%%%%%%%%%%%%%%%%%%%%%%%%%%%%%%%%%%%%%%%%%

\section{A grid on a cylinder}
\label{sec:AGridOnACylinder}

Let $C$ be a cycle and let $S \subseteq V(C)$ be a subset of its vertices. For $x,y \in V(C)$, let $A$ and $B$ be the two components (possibly empty) of $C-\{x,y\}$. We define the \emph{distance} between $x$ and $y$ on $C$ \emph{with respect to} $S$, denoted by ${\rm dist}_{(C,S)}(x,y)$  to be
$\min \{|V(A)\cap S|,|V(B)\cap S|\}$.

\begin{theorem}
\label{k6-on-c}
Let $G$  be a graph embedded in a cylinder, and suppose that $G$ has three pairwise disjoint homotopic cycles $C_1,C_2,C_3$, such that\/ $C_1$ and\/ $C_3$ coincide with the cuffs of the cylinder. Let $k \geq 7 $ be an integer and let $P_0,\dots,P_{k-1}$ be pairwise disjoint $(C_1,C_3)$-paths in $G$  such that for every $0 \leq i\leq k-1$ and\/ $1 \leq j \leq 3$, the intersection of $P_i$ and $C_j$ is a single vertex.
%\begin{itemize}
%	\item[$\circ$] For $i=0,\dots,k-1$, $P_i$ is an $(s_i,t_i)$-path, $s_i\in V(C_1)$ and $t_i \in V(C_3)$.
	%\item[$\circ$] The vertices $s_0,s_1,\dots,s_{k-1}$ (resp., $t_0,t_1,\dots,t_{k-1}$) appear on $C_1$ (resp., on $C_3$) in clockwise order.
%	\item[$\circ$] For $0 \leq i\leq k-1$, $1 \leq j \leq 3$, the intersection of $P_i$ and $C_j$ is a single vertex.
%\end{itemize}
For\/ $i=0,\dots,k-1$, let $s_i$ and $t_i$ be the ends of $P_i$ on $C_1$ and $C_3$, respectively.
Set\/ $S:= \{s_0,\dots,s_{k-1}\}$ and $T: = \{t_0,\dots,t_{k-1}\}$.
Let $a_1,a_2 \in V(C_1)$ and $b_1,b_2 \in V(C_3)$ such that $b_1 \neq b_2$ and\/
${\rm dist}_{(C_1,S)}(a_1,a_2) \geq 2$. Then the following holds:
\begin{enumerate}
	\item[\rm(i)] If\/ ${\rm dist}_{(C_3,T)}(b_1,b_2) \geq 1$, then $G'=G + a_1b_1 +a_2b_2$ contains a $K_6$-minor.
	\item[\rm(ii)] If\/ ${\rm dist}_{(C_3,T)}(b_1,b_2) =0$, and there exist vertices $a_3 \in V(C_1) $ and $b_3 \in V(C_3) $,  such that ${\rm dist}_{(C_3,T)}(b_1,b_3) \geq 1$ or ${\rm dist}_{(C_3,T)}(b_2,b_3) \geq 1$, then $G'=G + a_1b_1 +a_2b_2+ a_3b_3$ contains a $K_6$-minor.
\end{enumerate}
\end{theorem}

\begin{figure}[htb]
    \begin{center}
     \scalebox{0.7}{\includegraphics{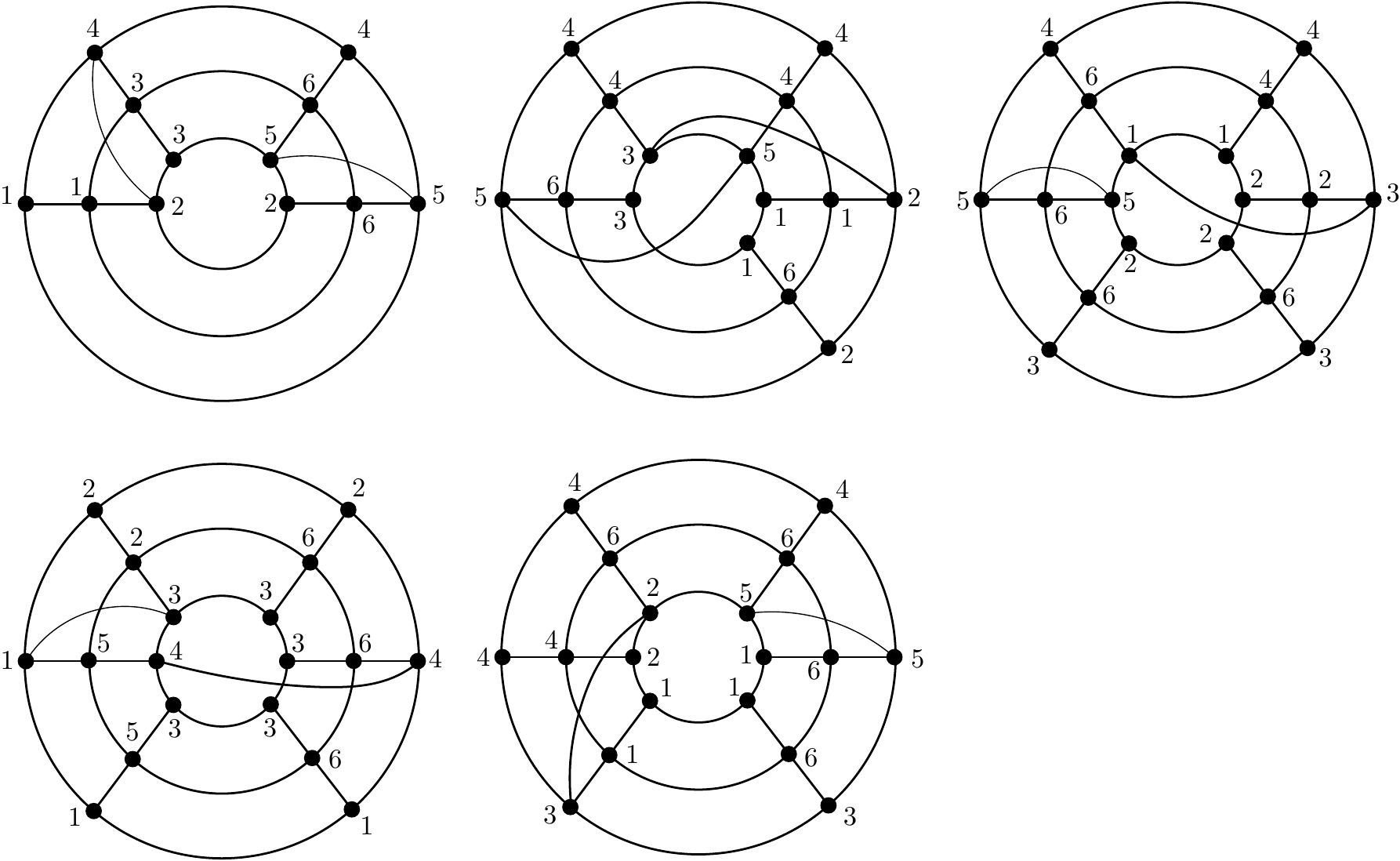}}
     \caption{The graphs $\Delta_1,\dots,\Delta_5$}
     \label{fig:fig1}
     \end{center}
\end{figure}

\begin{proof}
Throughout the proof all indices are taken modulo $k$. We introduce the graphs $\Delta_1,\dots,\Delta_5$ which are depicted in Figure~\ref{fig:fig1}. Each $\Delta_i$ contains a $K_6$-minor as is evident by the labelling of the vertices in the figure. Thus, it suffices to prove that one of these graphs is a minor of $G'$.

By relabelling the paths $P_i$ ($0 \leq i <k$), we may assume that the vertices $s_0,s_1,\dots,s_{k-1}$ (resp., $t_0,t_1,\dots,t_{k-1}$) appear on $C_1$ (resp., on $C_3$) in this cyclic order, and we consider the corresponding orientation of $C_1$ and $C_3$ as \emph{clockwise} orientation.
For $ i,j \in \{0,\dots,k-1\}$, let $C_1[s_i,s_j]$ (resp., $C_3[t_i,t_j]$) be the $(s_i,s_j)$-path (resp., $(t_i,t_j)$-path) on $C_1$ (resp., $C_3$) in the clockwise direction on the cycle. Set $\mathcal{C}:=C_1\cup C_2 \cup C_3 \subseteq G$.

By replacing $G$ with a minor of $G$, we may assume that $G$ is the union of the cycles $C_1,C_2,C_3$ and the paths $P_0,\dots,P_{k-1}$. Moreover, any contraction of an edge on $C_1$ or $C_3$ either identifies two vertices in $S \cup T$,  violates one of  distance assumptions on $a_1,a_2,b_1,b_2,b_3$, or identifies $b_1$ and $b_2$.

As a consequence, by (possibly) relabelling the paths, we may assume that the following conditions are satisfied:

\begin{enumerate}
\item[(1)] $a_1=s_0$, $a_2=s_j$ ($ 3 \leq j \leq k-3$).
\item[(2)] There are indices $0 \leq \ell < r \leq k-1$ such that $b_1=t_{\ell}$ and $b_2=t_{r}$ or $b_1=t_r$ and $b_2=t_{\ell}$.
\item[(3)] If (ii) holds, then $a_3 \in S$ and $b_3 \in T$.
\end{enumerate}

\noindent
\emph{Proof of} (i). We proceed according to three cases.

\smallskip

\textbf{Case 1.} Suppose that $ \ell,r  \in \{1,\dots,j-1\}$ or  $ \ell,r  \in \{j+1,\dots,k-1\}$. By symmetry (i.e. after possibly changing the labelling of the paths to the reverse cyclic labelling), we may assume the former. Since $r - \ell >1$, $s_{\ell+1} \neq s_r$.

If $b_1=t_{\ell}$ and $b_2=t_{r}$, then  $\Delta_1 \leqm \mathcal{C} \cup P_{j+1} \cup P_{\ell+1} \cup P_r \cup P_j \cup \{a_1b_1,a_2b_2\}$. To see this, consider the outer cycle of $\Delta_1$ to correspond to $C_1$ and the inner-most cycle to correspond to $C_3$. The four paths shown correspond in clockwise order, starting on the left,\footnote{We will stick with similar assumptions in the remaining cases: $C_1$ and $C_3$ will correspond to the outer and inner cycle, respectively, and the order of the paths in $\Delta_i$ will correspond to the listed order, starting on the left and continuing clockwise.}
to $P_{j+1},P_{\ell+1},P_r$ and $P_j$, and the two crossed edge are obtained by contracting $C_1(a_1,s_{\ell+1})$ and $C_3(t_{j+1},b_1)$.

If $b_1=t_{r}$ and $b_2=t_{\ell}$ then $\Delta_2 \leqm \mathcal{C} \cup P_0 \cup P_{\ell} \cup P_r \cup P_j \cup P_{j+1} \cup \{a_1b_1,a_2b_2\}$.

\smallskip

\textbf{Case 2.} Suppose $\{r,\ell\} \cap \{0,j\} \neq \emptyset$. By symmetry we may assume that $\ell=0$ and $ 2 \leq r \leq j$.

Suppose first that $r<j$. If $b_1=t_0$ and $b_2=t_r$, then $\Delta_1 \leqm \mathcal{C} \cup P_{j+1} \cup P_1 \cup P_r \cup P_j \cup \{a_1b_1,a_2b_2\}$ (this is obtained after contracting $C_1(a_1,s_1)$ and $C_3(t_{j+1},b_1)$). If $b_1=t_r$ and $b_2=t_0$, then $\Delta_2 \leqm \mathcal{C} \cup P_{k-1} \cup  P_1 \cup P_r \cup P_{j} \cup P_{j+1} \cup \{a_1b_1,a_2b_2\}$ (after contracting $C_1(s_{k-1},a_1)$ and $C_3(b_2,t_1))$.

Suppose now that $r=j$. Since $k\ge7$, we may assume by symmetry that $j\le k-4$. If $b_1=t_0$ and $b_2=t_j$ then $\Delta_1 \leqm \mathcal{C} \cup P_{k-1} \cup P_1 \cup P_{j-1} \cup P_{j+1} \cup \{a_1b_1,a_2b_2\}$.
If $b_1=t_j$ and $b_2=t_0$, then
$\Delta_2 \leqm \mathcal{C} \cup P_{k-1} \cup P_1 \cup P_{j-1} \cup P_{j+1} \cup P_{j+2} \cup \{a_1b_1,a_2b_2\}$
(we contract $C_1(s_{k-1},a_1)$, $C_1(a_2,s_{j+1})$, $C_3(b_2,s_1)$ and $C_3(t_{j-1},b_1)$.

\smallskip

\textbf{Case 3.} Suppose that $1 \leq \ell \leq j-1$ and $j+1 \leq r \leq k-1$. Then $\Delta_1 \leqm \mathcal{C} \cup P_r \cup P_0 \cup P_{\ell} \cup P_j \cup \{a_1b_1,a_2b_2\}$.

\bigskip

\noindent
\emph{Proof of} (ii). We have $r=\ell+1$ (since by assumption ${\rm dist}_{(C_3,T)}(b_1,b_2)=0$ and $b_1\ne b_2$). By symmetry, we may assume that $0 \leq \ell \leq j-1$.

\smallskip

\textbf{Case 1.} Suppose that $\ell=0$ or $\ell=j-1$. By symmetry we may assume that $\ell=0$ and then $r=1$.
If $b_1=t_0$ and $b_2=t_1$ then $\Delta_3 \leqm \mathcal{C} \cup P_0 \cup P_1 \cup P_{j-1} \cup P_{j} \cup P_{j+1} \cup P_{k-1} \cup \{a_1b_1,a_2b_2\}$. If $b_1=t_1$ and $b_2=t_0$ then $\Delta_4 \leqm \mathcal{C} \cup P_0 \cup P_1 \cup P_{j-1} \cup P_{j} \cup P_{j+1} \cup P_{k-1} \cup \{a_1b_1,a_2b_2\}$.

\smallskip

\textbf{Case 2.} Suppose that $1 \leq \ell \leq j-2$.
If $b_1=t_{\ell+1}$ and $b_2=t_{\ell}$, then $\Delta_2 \leqm \mathcal{C} \cup P_0 \cup P_{\ell} \cup P_{\ell+1} \cup P_j \cup P_{j+1} \cup \{a_1b_1,a_2b_2\}$.

Suppose now that $b_1=t_{\ell}$ and $b_2=t_{\ell+1}$. We may assume that $\ell=1$ and $\ell+2=j$. For if $\ell\neq 1$, then $\Delta_5 \leqm \mathcal{C} \cup P_1 \cup P_{\ell} \cup P_{\ell+1} \cup P_j \cup P_{j+1} \cup P_0 \cup \{a_1b_1,a_2b_2\}$, and the case when $\ell \neq j-2$ is symmetric to the case when $\ell\ne 1$.

\begin{figure}[htb]
    \begin{center}
     \scalebox{0.7}{\includegraphics{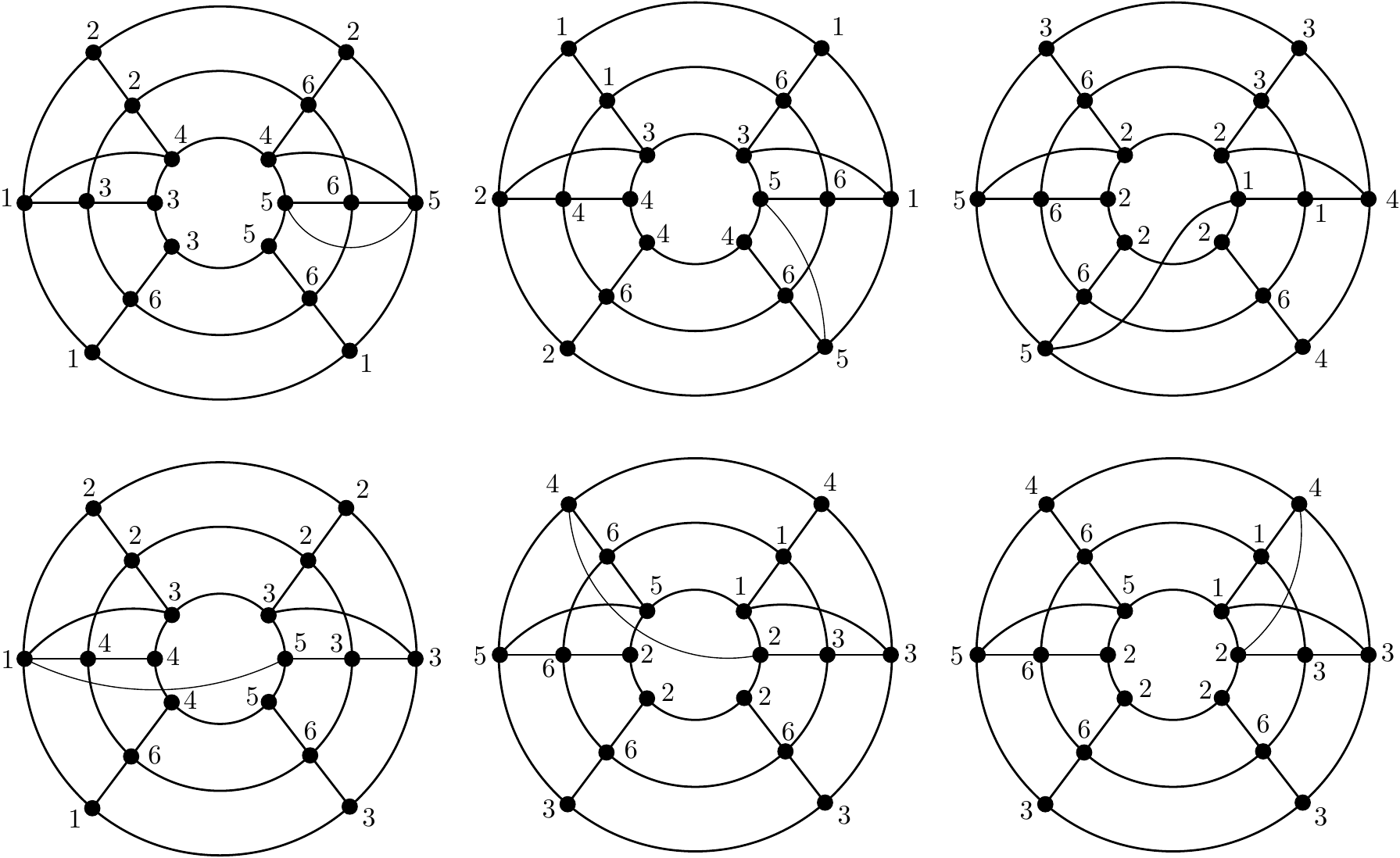}}
     \caption{The graphs obtained in Case~2 of proof of (ii) contain $K_6$-minors}
     \label{fig:fig2}
     \end{center}
\end{figure}

Hence, we are left with the case where $\ell=1$ and $j=3$.  By assumption, $b_3 \in T \setminus \{t_1,t_2\}$. Suppose that $b_3 \in \{t_5,\dots,t_{k-2}\}$. Then ${\rm dist}_{(C_3,T)}(b_3,b_1) \geq 2$ and ${\rm dist}_{(C_3,T)}(b_3,b_2)\geq 2$. In addition, since ${\rm dist}_{(C_1,S)}(a_1,a_2) \geq 2$, there is $z \in \{1,2\}$ such that ${\rm dist}_{(C_1,S)}(a_3,a_z) \geq 1$. Then the proof follows by  the proof of (i) by interchanging the roles of $C_1$ and $C_3$ and $S$ and $T$, with $b_3$ and $b_z$ playing the role of $a_1$ and $a_2$, and $a_3$ and $a_z$ playing the role of $b_1$ and $b_2$, respectively. Thus, we may assume that $b_3 \in \{t_0,t_3,t_4,t_{k-1}\}$.

Suppose that $b_3 \in \{t_0,t_3\}$. By symmetry, we may assume that $b_3=t_3$. Let $H:= \mathcal{C} \cup \{P_0,P_1,P_2,P_3,P_4,P_5\} \cup \{a_1b_1,a_2b_2,a_3b_3\}$.
By contracting edges on $C_1$, we obtain a minor $H'$ of $H$ such that $a_3 \in \{s_0,s_1,\dots,s_5\}$. For each of these six possibilities for $a_3$, we see that $H'$ contains a $K_6$-minor (see Figure~\ref{fig:fig2}).

Finally, suppose that $b_3 \in \{t_{4},t_{k-1}\}$. By symmetry, we may assume that $b_3 = t_{4}$.
Let $H:= \mathcal{C} \cup \{P_0,P_1,P_2,P_3,P_{k-2},P_{k-1}\} \cup \{a_1b_1,a_2b_2,a_3b_3\}$. Let $H'$ be obtained from $H$ by contracting $C_3(t_3,b_3)$. The proof now follows as in the previous paragraph.
\end{proof}

%============================
\section{The projective plane}

In this section we present the proof of Theorem \ref{thm:pp}.

The projective plane contains graphs of face-width 3 that do not contain $K_6$ as a minor. In fact the graphs obtained from $K_6$ by performing one or more $\Delta Y$-transformations\footnote{We say that a graph $H$ is obtained from a graph $G$ by a \emph{$\Delta Y$-transformation} if the edges of a triangle $T=uvw$ are removed from $G$ and replaced by a new vertex $y$ and three edges joining $y$ with each of $u,v,w$. The inverse operation is said to be a \emph{$Y\!\Delta$-transformation}.}
on facial triangles of $K_6$ provide such examples. On the other hand, face-width four forces $K_6$ minor as claimed by Theorem \ref{thm:pp}. In this section we give a proof of this theorem.

It suffices to prove Theorem \ref{thm:pp} for minor-minimal graphs embedded in the projective plane with face-width 4. It was proved by Randby \cite{Ra97} that every such graph can be obtained from the projective $4\times 4$ grid (the first graph depicted in Figure \ref{fig:4}) by a series of $Y\!\Delta$ and $\Delta Y$-transformations. Let ${\mathcal G}_4$ be the family of such graphs. It is known \cite{Ju95} that ${\mathcal G}_4$ contains precisely 270 graphs.

\begin{figure}[hbt]
    \begin{center}
     \scalebox{0.75}{\includegraphics{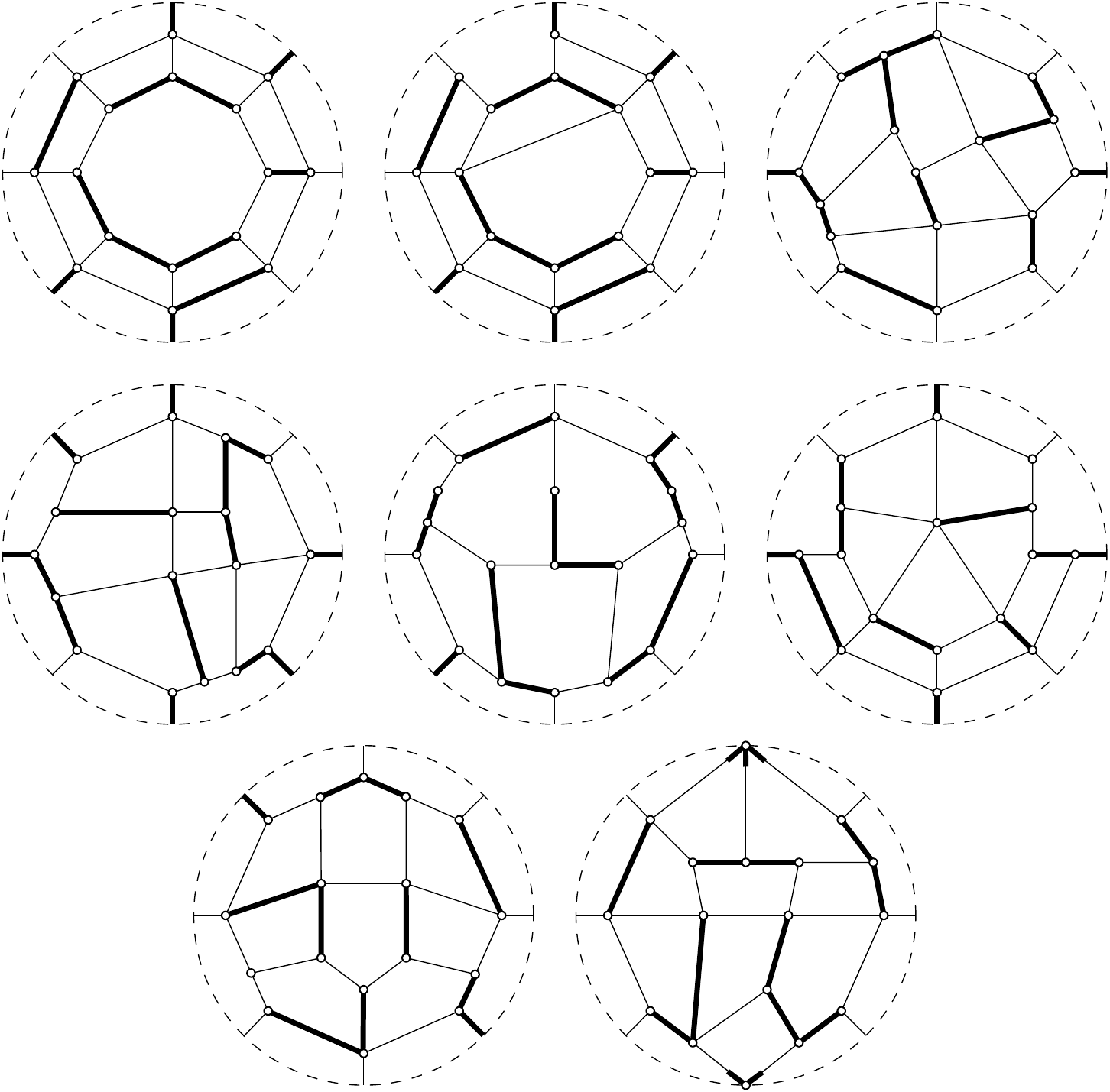}}
     \caption{The triangle-free graphs in ${\mathcal G}_4$ contain $K_6$ minors}
     \label{fig:4}
     \end{center}
\end{figure}

It is easy to see that if $G$ is obtained from $H$ by a $\Delta Y$-transformation and $G$ has a $K_6$ minor, then so does $H$. Therefore it suffices to prove that all triangle-free graphs in ${\mathcal G}_4$ contain $K_6$ as a minor. To justify this conclusion, note that graphs in ${\mathcal G}_4$ have face-width 4; thus every triangle in such a graph $G$ is facial in any embedding of $G$ on the projective plane. Also observe that every $\Delta Y$-transformation increases the number of vertices, thus performing these transformations as long as possible, we end up with a triangle-free graph.

Examining the 270 graphs in ${\mathcal G}_4$, we found that precisely eight of them are triangle-free. They are shown in Figure \ref{fig:4} (drawn in the projective plane), where also a $K_6$ minor is exhibited for each of them (the thick edges should be contracted in order to obtain a $K_6$ minor). This observation completes the proof of Theorem~\ref{thm:pp}.

%=================================
\section{Proof of the main result}
\label{sec:ProofOfMainResult}

In this section we prove Theorem \ref{thm:main}. Let $G$ and $\Sigma$ be as in the theorem.
By Theorem \ref{nsfw and large fw}, we may assume that $G$ is 3-connected and that $\fw(G)\ge6$. Let $\Lambda$ and $n$ be as in (H1). By Theorem~\ref{clean}, $\Lambda$ is clean. We shall also assume that $x_i \in V(G)$ ($i=0,\dots,n-1$), for if $x_i$ is an edge then we contract $x_i$ and work with the resulting minor of $G$. The only danger is that such contractions lower $\nsfw(G)$. However, this will not be a problem, since any further arguments involving large $\nsfw(G)$ will refer to the original graph.

Let $\Sigma$ be a cylinder with cuffs $F_1$ and $F_2$. Let $G$ be a graph embedded on $\Sigma$, and let $n\geq 2$ be an integer. Let $C_1,\dots,C_{n}$ be pairwise disjoint homotopic cycles in $G$ such that $F_1=C_1$, $F_2=C_n$  and $C_1,\dots,C_n$ appear along $\Sigma$ in order. For $i=1,\dots,n-1$, we say that $C_{i+1}$ is \emph{tight\/} in $G$ with respect to $C_i$, if there does not exist a $(C_{i+1})$-path $P$ that is disjoint from $C_{i+1}$ except for its two ends, $P$ is disjoint from $C_i$, and $P$ is embedded in the sub-cylinder of $\Sigma$ bounded by $C_i$ and $C_{i+1}$.

The proof proceeds according two two cases, depending whether $\Lambda$ is 2-sided or 1-sided.

\subsection{Proof of Theorem \ref{thm:main} when $\Lambda$ is 2-sided}
\label{sec:ProofOfRefLambdaIs2Sided}

As $\Gamma(\Lambda)$ is 2-sided, then as we traverse along $\Gamma(\Lambda)$ on $\Sigma$, one side is naturally the ``left-hand side'' and the other is the ``right-hand side''. The curve $\Gamma(\Lambda)$ splits each face $F_i$ into two closed disks. Each of these closed disks is bounded by the portion of $\Gamma(\Lambda)$ in $F_i$ and a part of the boundary of $F_i$. For $i=0,\dots,{n-1}$, let $\bnd_L(F_i)$ ($\bnd_R(F_i)$) be the portion of the boundary of $F_i$ to the left (right) of $\Gamma(\Lambda)$. Then each of $\bnd_L(F_i)$ and $\bnd_R(F_i)$ is a path in $G$ from $x_i$ to $x_{i+1}$ (indices modulo $n$). All these paths are pairwise disjoint except for their ends. Set $C_L(\Lambda):=\bigcup_{i=0}^{n-1} \bnd_L(F_i)$ and $C_R(\Lambda):=\bigcup_{i=0}^{n-1} \bnd_R(F_i)$. Since $\Lambda$ is clean, each of $C_L(\Lambda)$ and $C_R(\Lambda)$ is a cycle in $G$.

Cutting $\Sigma$ along $\Gamma(\Lambda)$, results in a new graph $G'$ embedded on $\Sigma'$, where $\Sigma'$ is the surface obtained from $\Sigma$ by cutting along $\Gamma(\Lambda)$ and capping off the resulting two cuffs. Let $F_L$ and $F_R$ be the two added faces of $G'$ whose boundaries coincide  with $C_L(\Lambda)$ and $C_R(\Lambda)$, respectively. By Theorem~\ref{newfacewidth}, we have $\nsfw(G')\geq 4$ and $\fw(G')\geq 3$, since $\nsfw(G)\ge 7$ and $\fw(G)\geq 6$. Let us now apply Theorem~\ref{fw-lemma-1} to $G'$ and its faces $F_L$ and $F_R$, respectively. Let $B_1(F_L)$, $C_1(F_L)$, $B_1(F_R)$ and $C_1(F_R)$ be the disks (cycles) as obtained by the application of Theorem~\ref{fw-lemma-1} and the facts that $\nsfw(G') \geq 4$ and $\fw(G')\geq 3$. Set $\Omega_R:=C_1(F_R)$ and $\Omega_L:=C_1(F_L)$. Note that $\Omega_L$ and $\Omega_R$ are homotopic to $\Gamma(\Lambda)$ and that they bound a cylinder containing $\Lambda$. By possibly altering $\Omega_R$ and $\Omega_L$, we may assume that the following holds:

\begin{lemma}
\label{tight}
In $G'$, the cycle $\Omega_L$ $($resp., $\Omega_R)$ is tight with respect to $F_L$ $($resp., $F_R)$, and\/ $\Omega_L \subseteq B_1(F_L)$ $($resp., $\Omega_R \subseteq B_1(F_R))$.
\end{lemma}

Next we observe that
\begin{lemma}
$\Omega_R$ and $\Omega_L$ are disjoint.
\end{lemma}

\begin{proof} For suppose not, then we let $v \in V(\Omega_R\cap \Omega_L)$. By the definition of $B_1(F_R)$ and $B_1(F_L)$, $v$ is co-facial with some vertex of $F_R$, say $w_R$, and some vertex of $F_L$, say $w_L$. In $G$, the vertices $v,w_R,w_L$, define a face-chain $\Lambda'$ of length two (not closed), starting and ending in $\Lambda$, whose faces are distinct from the faces of $\Lambda$. Let $\Lambda_1$ and $\Lambda_2$ be the two face-subchains\footnote{Strictly speaking, $\Lambda_1$ and $\Lambda_2$ need not be subchains since $w_L$ and $w_R$ need not be constituents of $\Lambda$. But all other faces and vertices in $\Lambda_1$ and $\Lambda_2$ are taken from $\Lambda$.}
in $\Lambda$ with ends $w_L$ and $w_R$. As $\Gamma(\Lambda')$ connects the left side of $\Gamma (\Lambda)$ with its right side, we see that both $\Gamma(\Lambda' \cup \Lambda_1)$ and  $\Gamma(\Lambda' \cup \Lambda_2)$ are surface non-separating in $G$.

To obtain a contradiction, note that by Theorem~\ref{maintool} (with $k=2$), one of $\Lambda' \cup \Lambda_1$ and $\Lambda' \cup \Lambda_2$
is of length at most 6. Since both $\Gamma(\Lambda' \cup \Lambda_1)$ and  $\Gamma(\Lambda' \cup \Lambda_2)$ are surface non-separating in $G$, we have a contradiction to the assumption that $\nsfw(G) \geq 7$.
\end{proof}

In $G$, the cycles $\Omega_R$ and $\Omega_L$  are homotopic to $\Gamma(\Lambda)$ (and homotopic to each other). Therefore, there exists $\Sigma' \subseteq \Sigma$ such that $\Sigma'$ is homeomorphic to a cylinder, the cuffs of which coincide with $\Omega_R$ and $\Omega_L$, and $\Gamma(\Lambda)\subseteq \Sigma'$. Let $G(\Omega_L,\Omega_R) \subseteq G$, be the subgraph of $G$ embedded in $\Sigma'$ (including $\Omega_L$ and $\Omega_R$).

Let $\mathcal{Q}=\{Q_1,Q_2,\dots\}$ be a set of pairwise disjoint paths, such that each $Q_i$ is an  $(\Omega_L,\Omega_R)$-path, disjoint from $G(\Omega_L,\Omega_R)$ except for its ends. If $\mathcal{Q}$ is of  maximum cardinality, then we say that $\mathcal{Q}$ is an \emph{exterior} $(\Omega_L,\Omega_R)$-linkage. By Lemma \ref{lem:disjoint-paths}, $|\mathcal{Q}| \geq \nsfw(G) \geq 7$.

By two applications of Lemma~\ref{lem:cylinder} and using Lemma~\ref{tight}, we see that  $G(\Omega_L,\Omega_R)$ contains a set $P_0^R,\dots,P_{n-1}^R$ and $P_0^L,\dots,P_{n-1}^L$ of pairwise disjoint paths, satisfying the following properties:
\begin{enumerate}
\item[(1)] For $i=0,\dots, n-1$, $P_i^L$ (resp., $P_i^R$) has ends $x_i$ and $l_i \in \Omega_{L}$ (resp., $r_i \in \Omega_{R}$) and is otherwise disjoint from $\Omega_L$ ($\Omega_R$) and $X(\Lambda)$. For $i=0,\dots,n-1$, set $P_i: = P_i^L \cup P_i^R$. Note that $P_i$ is an $(l_i,r_i)$-path contained in $G(\Omega_L,\Omega_R)$. Also note that the vertices $r_0,r_1,\dots, r_{n-1}$ ($l_0,l_1,\dots, l_{n-1}$) appear on $\Omega_R$ ($\Omega_L$) in order.
\item[(2)] Let $X\in \{L,R\}$. For $i=0,\dots,n-1$, $P_i$ is disjoint from $\bnd_X(F_j)$, if $j\neq\{i-1,i\}$ (indices modulo $n$). In addition, we may assume that $P_i-x_i$ intersects at most one of $\bnd_X(F_i)$ and $\bnd_X(F_{i-1})$. If $P_i -x_i$ intersects $F_{i}$ we say that $F_i$ is the  $X$-\emph{support} of $P_i$, otherwise $F_{i-1}$ is the $X$\emph{-support} of $P_i$. 	
\end{enumerate}

A set $\mathcal{P}=\{P_0,\dots,P_{n-1}\}$ of paths satisfying properties (1) and (2) above, is called an \emph{internal} $(\Omega_L,\Omega_R)$\emph{-linkage}. Figure \ref{fig:internallinkage} shows part of an internal linkage and the corresponding notation as used in the sequel.

For $i=0,\dots,n-1$, let $\Omega_R(i)$ (resp., $\Omega_L(i)$) be the path on $\Omega_R$ (resp., $\Omega_L$) from $r_i$ to $r_{i+1}$ (resp., $l_i$ to $l_{i+1}$) not passing thorough $r_{i+2}$ (resp., $l_{i+2}$).

\begin{figure}[hbt]
    \begin{center}
     \scalebox{1.0}{\includegraphics{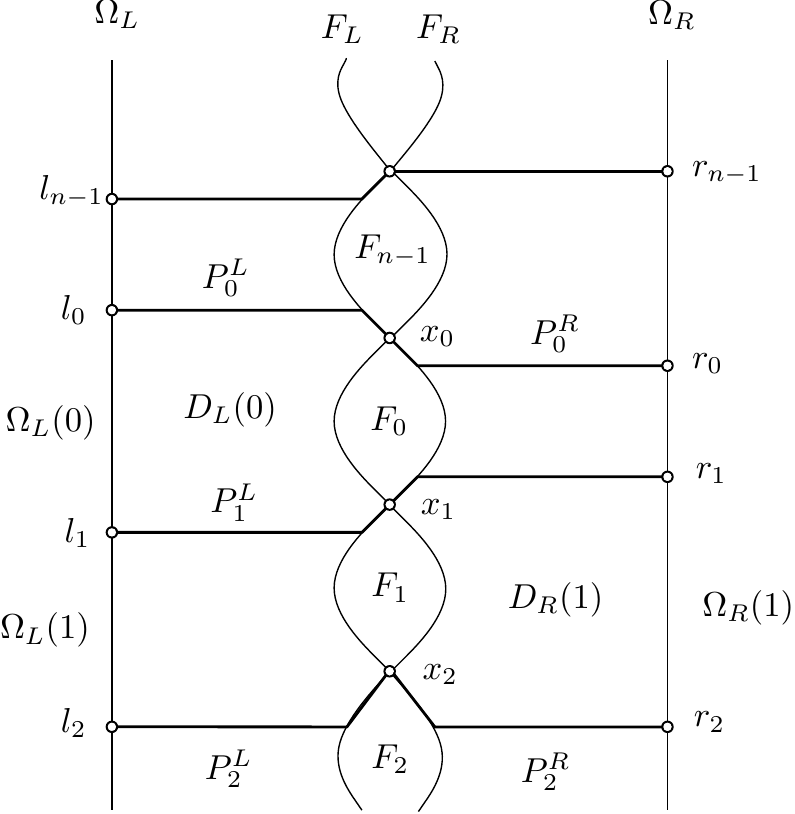}}
     \caption{Internal $(\Omega_L,\Omega_R)$-linkage}
     \label{fig:internallinkage}
     \end{center}
\end{figure}

Let $X \in \{R,L\}$. For a subset of indices $I=\{i_0,\dots,i_{|I|-1}\}\subseteq \{0,\dots,n-1\}$, let $\mathcal{F}_X(\mathcal{P},I) \subseteq F(\Lambda)$ be a set of consecutive faces of $\Lambda$ satisfying the following:
\begin{enumerate}
\item[(1)] For every vertex $v \in \cup_{i\in I} \Omega_X(i)$ there exists $f \in \mathcal{F}_X(\mathcal{P},I)$ such that $v$ is co-facial with some vertex of $f$.
\item[(2)] Subject to (1), $|\mathcal{F}_X(\mathcal{P},I)|$ is minimum.
\end{enumerate}
Observe that $\mathcal{F}_X(\mathcal{P},I)$ exists by Lemma~\ref{tight}.
Of special interest is the case when $I=\{i\}$ or $I=\{i,i+1\}$ for some $0 \leq i \leq n-1$. Note that $1 \leq |\mathcal{F}_X(\mathcal{P},\{i\})| \leq 3$ and $1 \leq |\mathcal{F}_X(\mathcal{P},\{i,i+1\})| \leq 4$.

For $i=0,\dots,n-1$ and $X\in \{R,L\}$, let $D_X(i)$ be the closed disk bounded by $\Omega_X(i)$, $P_i ^ X$, $P_{i+1} ^ X$ and a path in $C_X(\Lambda)$ on the boundary of the faces in $\mathcal{F}_X(\mathcal{P},\{i\})$.

The following is a direct consequence of the definition of $\mathcal{F}_X(\mathcal{P},\{i\})$.

\begin{lemma}
\label{setofsimpleclaims}
For $i=0,\dots,n-1$, each of $r_i$ and $r_{i+1}$ $($resp., $l_i$ and $l_{i+1})$ is co-facial in $G(\Omega_L,\Omega_R)$ with some vertex in $V(\mathcal{F}_R({\mathcal{P}},\{i\}))$ $($resp.,  $V(\mathcal{F}_L({\mathcal{P}},\{i\})))$.
\end{lemma}

A \emph{system} is a pair $(\mathcal{Q},\mathcal{P})$, where $\mathcal{Q}$ is an exterior $(\Omega_L,\Omega_R)$-linkage and $\mathcal {P}$ is an interior $(\Omega_L,\Omega_R)$-linkage. For $X \in \{L,R\}$ and a subset  $\mathcal{A} \subseteq \mathcal{Q} \cup \mathcal{P}$ of paths, we denote by  $Ends(\mathcal{A},\Omega_X)$ the set of endvertices of the paths in $\mathcal{A}$ contained in $\Omega_X$.
If $A$ is a single path, we set $End(A,\Omega_X) = Ends(\{A\},\Omega_X)$.
The following is the key ingredient in the proof.

\begin{lemma}
\label{distance-2}
Let $k \geq n \geq 7$, and let $\Im=(\mathcal{Q},\mathcal{P})$ be a system, where $\mathcal{Q}=\{Q_0, \dots, Q_{k-1}\}$ and $\mathcal{P}=\{P_0,\dots,P_{n-1}\}$. Let $Z \in \{L,R\}$ and $Y\in\{L,R\} \setminus Z$. Then there exists a system $(\mathcal{Q'},\mathcal{P'})$, where $\mathcal{Q'}=\{Q'_0, \dots, Q'_{k-1} \}$ and $\mathcal{P'}=\{P'_0,\dots,P'_{n-1} \}$ satisfying the following:
\begin{enumerate}
	\item[\rm(P1)] For $i=0,\dots,n-1$, $P_i^{'Z} = P_i ^ Z$.
	\item[\rm(P2)]  For $i=0,\dots,{k-1}$, $End(Q_i,\Omega_Z)=End(Q_i',\Omega_Z)$.
	\item [\rm(P3)] There exist paths $A, B\in \mathcal{Q'}$, such that $End(A,\Omega_Y)$ and $End(B,\Omega_Y)$ are at distance at least two on $\Omega_Y$ with respect to $Ends(\mathcal{P'},\Omega_Y)$.
\end{enumerate}
\end{lemma}

\begin{proof}
Assume for a contradiction that the claim is false. By symmetry we may assume that $Z=R$ and $Y=L$.

For every system $(\mathcal{Q}', \mathcal{P}')$ satisfying (P1) and (P2), let $S \subseteq \Omega_L$ be a minimal segment on the cycle $\Omega_L$ such that $Ends(\mathcal{Q}', \Omega_L) \subseteq V(S)$ and $S$ contains in its interior at most one vertex from the set $Ends(\mathcal{P}', \Omega_L)$. Note that $S$ exists since otherwise (P3) would hold. Let $T \subseteq \Omega_L$ be the minimal segment on $\Omega_L$ containing $S$ such that the endpoints of $T$ are in $Ends(\mathcal{P}',\Omega_L)$. Among all systems $(\mathcal{Q}', \mathcal{P}')$, we choose one with the following properties:

\begin{itemize}
	\item [\rm(J0)] $(\mathcal{Q}', \mathcal{P}')$ satisfies (P1) and (P2).
	\item [\rm (J1)] The number of vertices of $Ends(\mathcal{P}',\Omega_L)$ contained in $T$ is as large as possible and $|V(T)\setminus V(S)|$ is minimum.
\end{itemize}

Such a choice of $(\mathcal{Q}', \mathcal{P}')$ is clearly possible as $(\mathcal{Q},\mathcal{P})$ satisfies (P1) and (P2). By symmetry, we may assume that $T=\Omega_L(0)$ or $T=\Omega_L(0) \cup \Omega_L(1)$ and that $|Ends(\mathcal{Q},\Omega_L) \cap V(\Omega_L(0))| \geq 4$. To simplify notation, we may also assume that $(\mathcal{Q}, \mathcal{P}) = (\mathcal{Q}', \mathcal{P}')$.

Set $X:=\Omega_L(n-1) \cup \Omega_L(0) \cup \Omega_L(1)$. Then $X \subseteq \Omega_L$ is an $(l_{n-1},l_2)$-path on $\Omega_L$ containing $l_0$. Let $\overline{X}$ be the other $(l_{n-1},l_2)$-path on $\Omega_L$ such that $X \cup \overline{X} = \Omega_L$ and $X$ and $\overline{X}$ are disjoint except for their ends. We may further assume that $\Im'$ is chosen so that

\begin{itemize}
	\item [(J2)] Subject to (J1), $|V(X)|$ is minimum.
\end{itemize}

Observe that by the maximality of the number of disjoint paths in $\mathcal{Q}$, there does not exist an $(int(\overline{X}),\Omega_R)$-path which is internally disjoint from $G(\Omega_L,\Omega_R)$. Let us apply Theorem \ref{thm:disjoint-paths-2} to $G$ with $\mathcal{Q}$, $X$ and its ends (that is, the endpoints of $X$ playing the role of $w$ and $w'$ in Theorem \ref{thm:disjoint-paths-2}). Note that outcome (b) of Theorem \ref{thm:disjoint-paths-2} is obtained.
Let $u,v \in V(X)$ and $f$ be as promised to exist by Theorem \ref{thm:disjoint-paths-2}(b). Without loss of generality, assume that $v$ is closer to $l_{n-1}$ than $u$ on $X$.

The property of $f$ stated in Theorem \ref{thm:disjoint-paths-2} implies the following.

\begin{itemize}
\item[(1)] Every face-chain in $G(\Omega_L,\Omega_R)$ with ends $v$ and $u$ has length at least\/ $6$.
\end{itemize}
To see this, suppose that $\Lambda'$ is a face-chain in $G(\Omega_L,\Omega_R)$ of length at most five with ends $u$ and $v$. Then $\Lambda' \cup f$ is a closed face-chain in $G$ of length at most six, and by the property of $f$, $\Gamma(\Lambda' \cup f)$ is surface non-separating; contradicting the assumption that $\nsfw(G) \geq 7$. This proves~(1).

Property (1) immediately implies that

\begin{itemize}
	\item [(2)] $v \in V(\Omega_L(n-1))$ and $u \in V(\Omega_L(1))$.
\end{itemize}

Next we claim the following:

\begin{itemize}
	\item [(3)] In $G(\Omega_L,\Omega_R)$ there exist face chains $g_v,g_u$ of length at most two such that $g_v$ has ends $v,l_0$, and $g_u$ has ends $l_1,u$. Moreover, if $g_v$ (resp., $g_u$) is of length two, then $v$ (resp., $u$) is co-facial with some vertex in the $L$-support of $P_0$ (resp., $P_1$).
	\end{itemize}
We will prove existence of $g_u$ (the proof for $g_v$ is exactly the same; in fact it is even easier since  no $Q \in \mathcal{Q}$ has an end in the interior of $\Omega_L(n-1)$).

Let $ j \in \{0,1\}$ so that $F_j$ is the $L$-support of $P_1$. Let  $w$ be the end of the path $P_1 \cap \bnd_L(F_j)$ of $P_1$ so that $w \neq x_1$, unless the path $P_1 \cap \bnd_L(F_j)$ is the single vertex $x_1$. It suffices to show that every vertex $x \in \Omega_L(l_1,u)$ is of degree two in $G(\Omega_R, \Omega_L)$ or has no neighbors in $G(\Omega_R, \Omega_L)$ except for $w$ and the two neighbors of $x$ on the cycle $\Omega_L$. (Note that $x$ is adjacent in $G(\Omega_L, \Omega_R)$ with only two vertices in $\Omega_L$, since $\Omega_L$ is tight.)

Suppose to the contrary that there exists $x \in \Omega_L(l_1,u)$ such that $x$ has a neighbor in $G(\Omega_L,\Omega_R)$ that is distinct from $w$ and from the two neighbors of $x$ on the cycle $\Omega_L$.
Since $G$ is 3-connected (and hence $G-w$ is 2-connected), by Menger's theorem there exists a $(\partial D_L(1))$-path $P$ with ends  $x$ and $y$, where $y\in V(\partial D_L(1)) \setminus \{w\}$. Since $\Omega_L$ is tight, we have $y\in V(F_L)$.

\textbf{Case 1.} Suppose $j=1$. By Lemma~\ref{tight}, $y \in V(F_1 \cup F_2) \setminus \{w\}$.
%\partialto some vertex, say $y$, embedded in the closed disk $D_1$. By~\ref{tight}, $P$ is disjoint from $(P_1 \downarrow L) - (P_1 \cap F_1)$, $(P_2 \downarrow L) - (P_2 \cap (F_1 \cup F_2))$ and from $\Omega_L[l_1,l_2]$ except for its end $x$. In particular, $P$ is disjoint from $P_1$ (it is disjoint from $w$ by definition of $P$) and $y \in F_1 \cup F_2$.
Let $P_2'$ be the path obtained from $P_2$ by rerouting $P_2 ^ L$ so that it passes via $P$. Let $\mathcal{P'}$ be the new collection of paths. We claim that $\mathcal{P'}$ contradicts our choice of $\{\mathcal{Q},\mathcal{P}\}$.

If $Ends(\mathcal{Q},\Omega_L(l_1,l_2]) = \emptyset$, then $\{\mathcal{Q},\mathcal{P'}\}$ contradicts  (J2). Suppose that $Ends(\mathcal{Q},\Omega_L(x,u]) \allowbreak \neq \emptyset$ for some $Q \in \mathcal{Q}$, and let $Q'$ be a path in $\mathcal{Q}$ with an end in $\Omega_L(l_0,l_1)$. Such a path exists since $|Ends(\mathcal{Q},\Omega_L(0))| \geq 3$ and hence $|Ends(\mathcal{Q},\Omega_L[l_0,l_1)| \geq 1$. Then $Q$ and $Q'$ are at distance at least two with respect to $Ends(\mathcal{P'},\Omega_L$) and hence $\{\mathcal{Q}, \mathcal{P}'\}$ satisfies (P3).
It follows that $Ends(\mathcal{Q},\Omega_L) \subseteq \Omega_L(0) \cup \Omega_L[l_1,x]$,  but then $\{\mathcal{Q}, \mathcal{P}' \}$ contradicts (J1).

\textbf{Case 2.} $j=0$. By Lemma~\ref{tight}, $y \in V(F_0 \cup F_1 \cup F_2) \setminus  \{w\}$. We first observe  that $u$ is not co-facial in $G(\Omega_L,\Omega_R)$ with any vertex in $V(F_0)$. For suppose $u$ is co-facial in $G(\Omega_L,\Omega_R)$ with $F_0$, say via a face $g_0$. By Lemma~\ref{setofsimpleclaims}, $v$ is co-facial in $G(\Omega_L,\Omega_R)$, say via a face $g_1$, with some vertex of $V(\bnd_L(F_i))$, for some $i\in \{n-2,n-1,0\}$. Then using $g_0,g_1,F_{n-2},F_{n-1}$ and $F_0$ we can construct a face-chain of length at most five in $G(\Omega_L,\Omega_R)$ with ends $u$ and $v$, contradicting~(1).

If $y \in V(F_1 \cup F_2)\setminus\{x_1\}$, the proof proceeds exactly as in Case~1 (by replacing $P_2^L$ by another path using $P'$ and thus obtaining a contradiction).
Hence, we may assume that $y \in V(F_0) \setminus \{w\}$.
Since $u$ is not co-facial with any vertex in $V(F_0)$, Lemma~\ref{tight} implies that there exists a path $P'$ in $D_L(1)$ with one end in $V(F_1 \cup F_2)  \setminus \{x_1\}$ and the other end in $V((P-F_0) \cup \Omega_L[x,u))$. We then see that there exists a  path $P''$ with one end in $V(\Omega_L[x,u))$ and one end in $V(F_1 \cup F_2) \setminus \{x_1\}$, and the proof again proceeds exactly as in Case~1 with $P''$ playing the role of $P$. This proves (3).

\medskip

Observe that at least one of $P_0$ and $P_1$ is not $L$-supported by $F_0$. For if both $P_0$ and $P_1$ are $L$-supported by $F_0$, then by Lemma \ref{setofsimpleclaims}, each of $l_0$ and $l_1$ is co-facial  with some vertex in $V(F_0)$. By (3), there is a face-chain of length at most two from $v$ (resp., $u$) to a vertex in $V(F_0)$, since by (3), $v$ is either co-facial with $l_0$ (resp., $u$ is co-facial with $l_1$) in $G(\Omega_L,\Omega_R)$ or co-facial with a vertex in $V(F_0)$. Combining these faces together with $F_0$, we obtain a face-chain from $v$ to $u$ in $G(\Omega_L,\Omega_R)$ that is of length at most 5, contradicting~(1).

We will assume henceforth that $P_1$ is $L$-supported by $F_1$ (if $P_1$ is $L$-supported by $F_0$, the proof follows the same arguments). Now we distinguish two cases: either $P_0$ is $L$-supported by $F_0$ or by $F_{n-1}$. We consider the latter case (the former case is proved by the same arguments).

By~(1) and (3), $l_0$ and $l_1$ are not co-facial in $G(\Omega_R,\Omega_L)$. Hence there exists $x \in V(\Omega_L(l_0,l_1))$ of degree at least three in $G(\Omega_R,\Omega_L)$. By Menger's theorem and since $\Omega_L$ is tight, there is $(\partial D_L(0))$-path $P$ in $D_L(0)$ with ends $x$ and $y$, where $y \in \partial D_L(0) \cap( F_{n-1} \cup F_0 \cup F_1)$.

\textbf{Case 1.} Suppose that $y \in V(F_0) \setminus \{x_0,x_1\}$. Let $P_0'$ (resp., $P_1'$) be obtained from $P_0$ (resp., $P_1$) by re-rerouting it so that it passes via $P$ rather via $P_0^L$ (resp., $P_1^L$).
Let $\mathcal{P'}= (\mathcal{P} \setminus \{P_0\}) \cup \{P_0'\}$ and let $\mathcal {P''} = (\mathcal{P} \setminus \{P_1\}) \cup \{P_1'\}$. We claim that one of  $(\mathcal{Q},\mathcal{P'})$ or $(\mathcal{Q},\mathcal{P''})$ contradicts our choice of $(\mathcal{Q},\mathcal{P})$. We argue as follows.

We may assume that  $Ends(\mathcal{Q},\Omega_L[l_0,x)) \neq \emptyset$, for otherwise  $(\mathcal{Q},\mathcal{P'})$ contradicts (J1). Further, we may assume that $Ends(\mathcal{Q},\Omega_L(l_1,l_2]) =\emptyset$, for otherwise $(\mathcal{Q},\mathcal{P'})$ satisfies (P3) (since $Ends(\mathcal{Q},\Omega_L[l_0,x)) \neq \emptyset$). Now it is easy to see that $(\mathcal{Q},\mathcal{P''})$ contradicts one of the two conditions stated in (J1).

\textbf{Case 2.} Suppose  that $y \in V(F_1 \cup F_{n-1})$. We may assume that $y \in V(F_1)$ (if $y \in V(F_{n-1})$ the proof follows by the same arguments). Let  $w$ be the end of the path $P_0 \cap \bnd_L(F_{n-1})$ so that $w \neq x_0$ unless $V(P_0 \cap \bnd_L(F_{n-1})) = \{x_0\}$.  We proceed according to two cases, depending on whether $l_0$ is co-facial with a vertex in $V(F_1)$ or not.

\textbf{Case 2.1.} If $l_0$ is not co-facial with a vertex in $V(F_1)$, there exists a path $R$ in $D_L(0)$ with one end in $V(\Omega_L(l_0,x])$ and the other in $V(F_{n-1} \cup (F_0 -x_1))$ (since $\Omega_L$ is tight). Let $P_0'$ (resp., $P_1'$) be obtained from $P_0$ (resp., $P_1$) by re-rerouting it so that it passes via $R$ (resp., $P$) rather than via $P_0^L$ (resp., $P_1^L$). The proof now proceed exactly as in Case (1).

\textbf{Case 2.2.} Suppose that $l_0$ is co-facial with some vertex in $V(F_1)$ via a face $f_0$.  By (1) and (3), $g_v$ must be a face-chain of length two. By (the proof of) (3), $v$ and $w$ are co-facial in $G(\Omega_L,\Omega_R)$ and there exists a vertex $z \in \Omega_L(v,l_0)$ such that $zw \in E(G(\Omega_L,\Omega_R))$. Hence by (1), $w$ is not co-facial with a vertex in $V(F_1)$. Thus, there exists a path $R$ in $D_L(0)$ with one end in $V(F_{n-1} \cup (F_0 -x_1))$ and the other end, say $a$, in $V(P_0^L - V(F_{n-1}))$. Note that the path $R$ cannot end up in $\Omega_L(0)\setminus\{l_0\}$ because of the face $f_0$.

Let $P_0'$ (resp., $P_{n-1}'$) be obtained from $P_0$ (resp., $P_{n-1}$) by re-rerouting it so that it passes via $R$ (resp., the edge $wz$) rather than via $P_0^L$ (resp., $P_{n-1}^L$). Let $\mathcal{P}'$ be the new collection of paths obtained by replacing $P_0$ and $P_{n-1}$ with $P_0'$ and $P_{n-1}'$. Then $(\mathcal{Q}, \mathcal{P}')$ contradicts (J2). This completes the proof of Lemma~\ref{distance-2}.
\end{proof}

Now we can complete the proof of Theorem \ref{thm:main} when $\Lambda$ is 2-sided. Let $(\mathcal{Q},\mathcal{P})$ be a system. By two applications of Lemma~\ref{distance-2}, we may assume that $(\mathcal{Q},\mathcal{P})$ satisfies (P3)  for $X \in \{L,R\}$.

Let $A,B \in \mathcal{Q}$, such that $End(A,\Omega_L)$ and $End(B,\Omega_L)$ are at distance at least two with respect to $Ends(\mathcal{P},\Omega_L)$. Set $q_A=End(A,\Omega_R)$ and $q_B=End(B,\Omega_R)$.
If $q_A$ and $q_B$ are at distance at least one on $\Omega_R$ with respect to $Ends(\mathcal{P},\Omega_R)$, set $H=G(\Omega_L,\Omega_R) \cup A \cup B$. Otherwise, $q_A,q_B \in V(\Omega_R(i))$ for some $0 \leq i \leq n-1$. By~(P3), there exists a path $C \in \mathcal{Q}$ such that $End(C,\Omega_R) \not\in V(\Omega_R(i))$. Set  $H=G(\Omega_L,\Omega_R) \cup A \cup B \cup C$.

For each $i\in \{0,\dots,n-1\}$, $P_i$ intersect each of $\Omega_L$ and $\Omega_R$ in a single vertex. In addition, by definition, $P_i' = P_i\cap G(\Lambda)$ is a sub-path of $P_i$ with $x_i \in V(P_i')$.
Let $H_1$ be obtained  from $H$ by contracting the path $P_i'$ into the vertex $x_i$, for $i=0,\dots,n-1$.
By Theorem \ref{k6-on-c}, $H_1$ contains a $K_6$ minor, and hence also $G$. This completes the proof.

\subsection{Proof of Theorem \ref{thm:main} when $\Lambda$ is 1-sided}

Since the deletion of any vertex decreases the non-separating face-width at most by 1, we may assume that $\nsfw(G)=7$. By Theorem \ref{nsfw and large fw}, we may also assume that $G$ is 3-connected and that $\fw(G)\ge6$. Moreover, we shall assume throughout this subsection that $\Lambda$ is 1-sided.
Let $G'$ be the embedded graph obtained from $G$ by cutting the surface along $\Gamma(\Lambda)$ and capping off the resulting cuff with a disk $F$. In $G'$, every vertex $x_i \in X(\Lambda)$ ($i=0,1,\dots,6$) is split into two copies, $x_i'$ and $x_i''$, and the vertices $x_0',x_1',\dots,x_6',x_0'',x_1'',\dots,x_6''$ appear on the boundary of $F$ in the listed cyclic order.
By Theorem \ref{thm:pp} we may assume that $\Sigma$ is not the projective plane, thus the resulting surface $\Sigma'$ is not the sphere.
By Theorem~\ref{newfacewidth}, $\nsfw(G') \geq \left\lceil \tfrac{1}{2}\nsfw(G) \right\rceil \geq 4$ and $\fw(G')\ge3$.

Let us first consider the possibility that $\fw(G')=3$. Let $\Gamma'$ be the corresponding non-contractible curve. Clearly, $\Gamma'$ involves the face $F$ and two other faces $A,B$ that are also faces of $G$ in $\Sigma$. Let $x,F,y,B,z,A,x$ be the corresponding closed face-chain in $\Sigma'$. The curve $\Gamma'$ is surface-separating on $\Sigma'$ since $\nsfw(G')\ge4$. We can view $\Gamma'$ as a simple closed curve $\Gamma_0$ in $\Sigma$ by replacing the part of $\Gamma'$ in $F$ by a segment of $\Gamma(\Lambda)$. We can do this in two ways, so we may take a segment of $\Gamma(\Lambda)$ such that $\Gamma_0$ intersects at most 3 vertices in $X(\Lambda)$ that are different from $x$ and $y$. Since $\nsfw(G)\ge7$, $\Gamma_0$ is surface-separating in $\Sigma$ (possibly contractible) and it
separates the surface into two non-spherical surfaces $\Sigma_1$ and $\Sigma_2$. One of them, say $\Sigma_1$, contains a 1-sided curve corresponding to $\Lambda$. The part of $\Gamma'$ disjoint from the interior of the face $F$ can be combined in $\Sigma$ with two segments contained in $\Lambda$ to give two closed curves. One of them is $\Gamma_0$, and we call the other one $\Gamma_1$. Observe that $\Gamma_1$ is homologous to $\Lambda$ (thus 1-sided) and $\Gamma_0$ is surface-separating in $\Sigma$. Since $\nsfw(G)\ge7$ and $\fw(G)\ge6$, $\Gamma_1$ necessarily passes through four consecutive vertices in $X(\Lambda)$, and $\Gamma_0$ passes through the remaining three vertices in $X(\Lambda)$. We may assume that $\Gamma_0$ passes through $x_0,x_1,x_2$ and through the vertices $x,y,z$. One particular observation is that $x,y\notin X(\Lambda)$ and that the face-chain of $\Gamma_0$ is $x,A,z,B,y,F_2,x_2,F_1,x_1,F_0,x_0,F_6,x$ (see Figure \ref{fig:Gamma2prime}).

\begin{figure}[hbt]
    \begin{center}
     \scalebox{1.1}{\includegraphics{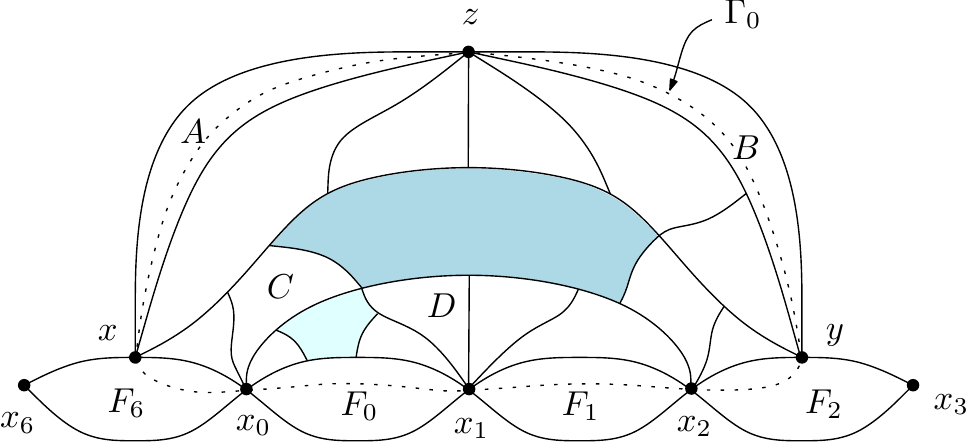}}
     \caption{The surface $\Sigma_2'$}
     \label{fig:Gamma2prime}
     \end{center}
\end{figure}

Let us consider the set $\F$ of faces of $G$ that lie inside $\Sigma_0$ and are incident with vertices in $X(\Gamma_0)$. Each such face is incident with precisely one vertex in $X(\Gamma_0)$. If not, we would either contradict 3-connectivity of $G$ or the fact that $\fw(G)\ge6$. If a face $Q\in\F$ is incident with $t\in X(\Gamma_0)$, we say that $Q$ is a $t$-face, and we let $\F_t$ denote the set of all $t$-faces in $\F$.

We say that two distinct vertices $s,t\in X(\Gamma_0)$ are at \emph{distance $d$\/} if their minimum face-distance in the closed face-chain of $\Gamma_0$ is equal to $d$. Note that $d\in\{1,2,3\}$.

Suppose that $s,t\in X(\Gamma_0)$ are at distance 3 and that $A\in\F_s$, $B\in\F_t$.
If $A$ and $B$ have a vertex $v$ in common, then the face-chain $s,A,v,B,t$ and the two face-subchains of $\Gamma_0$ give rise to two closed face-chains in $\Sigma$ of length 5, so they determine contractible closed walks. The 3-path-property implies that $\Gamma_0$ is also contractible. This contradiction shows that $A\cap B=\emptyset$.

If $s,t\in X(\Gamma_0)$ are at distance 2 and $A\in\F_s, B\in\F_t$ have a vertex $v$ in common, then we similarly see that one of the face-chains in $\Sigma$ obtained in the same way as above is of length 4, the other one of length 6. The first one determines a contractible curve in $\Sigma$. We can re-route $\Gamma_0$ through $v$, thus making $\Sigma_2'$ smaller. By repeating this process as long as necessary, we may assume that faces in $\F_s$ and $\F_t$ are disjoint whenever $s$ and $t$ are at distance 2.

If $s$ and $t$ are at distance 1 and two faces, $C\in\F_s$ and $D\in\F_t$, have a vertex $v$ in common (e.g.\ the faces $C,D$ depicted in Figure \ref{fig:Gamma2prime}), then there is a face-chain of length 3 through $s,t,v$ and the two faces. The corrersponding closed curve $\Gamma$ in $\Sigma$ is contractible, and we add all faces in the interior of $\Gamma$ into $\F$. After doing this for all possible choices of $s,t,C,D$, we define $\partial\F$ as the set of edges that belong to precisely one face in $\F$ and do not belong to any of the faces of $\Gamma_0$. The properties stated in the preceding paragraphs imply that $\partial \F$ is a simple cycle in $G$ that is homotopic to $\Gamma_0$. (In Figure \ref{fig:Gamma2prime}, this cycle is represented as the boundary of the darker shaded area. All faces in the lighter shaded area belong to $\F$ and form a disk in $\Sigma$.) Now we delete all edges and vertices in $\Sigma_2$ that do not belong to any of the faces in $\F$ and cap off the cycle $\partial \F$ by pasting a disk onto it. This gives rise to a subgraph $G_1$ of $G$ embedded into the capped surface $\Sigma_1$. It is easy to see by using the 3-path-property that $\nsfw(G_1)\ge7$ since every surface non-separating face-chain through the disk of $\partial \F$ can be rerouted to use the face-chain of $\Gamma_0$ without increasing its length.
Since the genus decreases by the reduction from $\Sigma$ to $\Sigma_1$, such a reduction can be made only a finite number of times, eventually yielding a case where $\fw(G')\ge4$.

From now on, we shall assume that $\fw(G')\ge 4$.
Let us apply Theorem \ref{fw-lemma-1} to the embedding of $G'$ in $\Sigma'$ and the face $F$. For $i=0,1$, let $C_i=C_i(F)$ be the cycle as in Theorem \ref{fw-lemma-1}. Since $\fw(G')\ge4$, these two cycles are contractible in $\Sigma'$.

The boundary of $F$ is a cycle in $G'$. In $G$, it corresponds to a closed walk which intersects itself transversally when passing through the vertices in $X(\Lambda)$, but it does not cross itself on the surface. In this sense we view $C_0$ as a closed walk in $G$. Theorem \ref{fw-lemma-1} assures that $C_0$ is homotopic to $C_1$.

Consider the cycle $C_1$ in $G$. Cutting $G$ along $C_1$ separates $\Sigma$ into two components, one of which contains $\Lambda$ and $C_0$. This surface is homeomorphic to the M\"obius strip. By capping off the cuff (pasting a disk onto $C_1$), we obtain a graph embedded into the projective plane $\Sigma_1$. We denote by $F_1$ the face in $\Sigma_1$ bounded by the cycle $C_1$. We also denote by $\Sigma_2$ the other bordered surface obtained after cutting $\Sigma$ along $C_1$.

Let $v\in V(C_1)$. Since $G$ is 3-connected and the embedding of $G$ in $\Sigma$ has face-width more than 3, the facial neighborhood of $v$ forms a disk on the surface that is bounded by a cycle $N_v$. This cycle contains a path $P_v$ whose ends $x,y$ are on $C_1$ but all edges and other vertices on this path lie in $\Sigma_2\setminus C_1$. Moreover, $P_v$ can be selected so that the cycle $Q_v$ consisting of $P_v$ and the $(x,y)$-segment of $C_1$ containing $v$ is contractible in $\Sigma_2$, and the interior of $Q_v$ contains all faces that are incident with $v$ and are contained in $\Sigma_2$. (The proof of this fact is essentially the same as the main argument in the proof of Theorem \ref{fw-lemma-1}; cf.\ \cite{Mo92} or \cite{mohar}.) If $u,v\in V(C_1)$ and the ends of the paths $P_v$ and $P_u$ interlace on $C_1$, contractibility of the cycles $Q_v$ and $Q_u$ implies that $P_v$ and $P_u$ intersect. This property has the following consequence. Let $H'$ be the minor of $G\cap \Sigma_2$ obtained from $P = \cup_{v\in V(C_1)} P_v \cup C_1$ by contracting all edges in $P$ whose both ends are outside $C_1$. Then $H'$ consists of $C_1$ together with some chords of $C_1$ and some vertices whose all neighbors lie on $C_1$. The interlacing property stated earlier implies that $H'$ can be drawn in the disk so that $C_1$ is on the boundary of the disk. By inserting this disk into the face $F_1$ in $\Sigma_1$ we obtain a minor $G''$ of $G$ that is embedded into the projective plane. It is easy to see that the face-width of $G''$ is at least 4, and by Theorem \ref{thm:pp}, $G''$ contains $K_6$ as a minor. Since $G''$ is a minor of $G$, we conclude that $G$ has $K_6$ minor. This completes the proof of Theorem \ref{thm:main}.

\section*{Acknowledgement}We are grateful to Lino Demasi who provided us with the list of all triangle-free graphs in the $Y\!\Delta$-class of the projective $4\times 4$ grid.

\bibliographystyle{abbrv}
\bibliography{refs}

\end{document}